\documentclass[11pt]{article}

\usepackage[margin=1in]{geometry}
\usepackage{xcolor}
\usepackage{amsmath,amssymb,amsthm,mathtools}
\usepackage{stmaryrd}
\usepackage{enumitem}
\usepackage{hyperref}
\usepackage{mathrsfs}
\usepackage[nameinlink,capitalise,noabbrev]{cleveref}
\usepackage{amsmath,amssymb}
\usepackage{xcolor}
\usepackage{tikz}
\usepackage{float}

\usetikzlibrary{
  arrows.meta,
  calc,
  positioning,
  decorations.pathreplacing
}

\definecolor{fbblue}{RGB}{44, 107, 160}
\definecolor{fbgray}{RGB}{105, 105, 105}
\definecolor{fbred}{RGB}{190, 65, 60}
\definecolor{fbgreen}{RGB}{55, 135, 95}
\tikzset{
  maincurve/.style={line width=.9pt, fbblue},
  auxiliary/.style={line width=.55pt, dashed, fbgray},
  maparrow/.style={-{Stealth[length=2.2mm]}, line width=.55pt, fbblue},
  lightarrow/.style={-{Stealth[length=1.8mm]}, line width=.45pt, fbgray}
}

\hypersetup{
  colorlinks=true,
  linkcolor=blue!50!black,
  citecolor=blue!50!black,
  urlcolor=blue!50!black
}

\newtheorem{theorem}{Theorem}[section]
\newtheorem{proposition}[theorem]{Proposition}
\newtheorem{lemma}[theorem]{Lemma}
\newtheorem{corollary}[theorem]{Corollary}
\theoremstyle{definition}

\newtheorem{remark}[theorem]{Remark}

\newtheorem{question}[theorem]{Question}

\title{Local Variational Properties of Non-degenerate Free Boundary Minimal Hypersurfaces}

\author{Xiaoxiang Jiao \and Qinhan Zhao \and Hangyue Zhu}
\date{\today}

\begin{document}

\maketitle

\begin{abstract}
We establish local variational characterizations of non-degenerate free-boundary minimal hypersurfaces in compact Riemannian manifolds with boundary.

For a two-sided strictly stable hypersurface, we show that it is the unique mass minimizer in its relative homology class, both in a small tubular and flat-neighborhood. As an application, for generic metrics in dimensions \(3\le n+1\le 7\), we show that the first free-boundary width is realized by a multiplicity-one hypersurface of Morse index one.

For an orientable non-degenerate free-boundary hypersurface of Morse index \(k>0\), we construct a canonical local \(k\)-parameter family and obtain a local min--max characterization.
\end{abstract}

\tableofcontents

\section{Introduction}

\subsection{Background and Main Result}
A basic principle in finite-dimensional variational theory is that a strictly stable critical point is a strict local minimizer. More precisely, if \(f\) is a smooth function on a finite dimensional Banach space and \(x_0\) is a critical point such that the second variation of \(f\) at \(x_0\) is positive definite, then, from Taylor expansion we have:
\[
f(x)>f(x_0)\qquad \text{for every \(x\neq x_0\) sufficiently close to \(x_0\).}
\] 

However, the analogous question for geometric variational problems is subtler. Suppose that \(\Sigma\) is a strictly stable critical point of the area functional, it is natural to ask the following question:
\begin{question}
\label{q1}
Does strict stability of a critical submanifold imply that it is a strict local minimizer among all admissible currents in a sufficiently small neighborhood?
\end{question}
The difficulty in Question~\ref{q1} lies in the gap between stability and local minimality. Strict stability is detected by the second variation and is therefore a statement about regular nearby deformations. However, the natural competitors in geometric measure theory may carry singularities or multiplicity. A fundamental result of B. White gives a positive answer in closed and fixed boundary settings {\cite[Theorem 2]{White1994}}. He proved that a strictly stable submanifold with possibly empty boundary is locally unique homologous minimizer for elliptic functionals.

However, White's theorem does not directly address the free-boundary situation, where the boundary of the competitor is allowed to move along the ambient boundary. This leads to the following free-boundary analogue of Question \ref{q1}:

\begin{question}
\label{q2}
In the free-boundary setting, does strict stability characterize a minimal submanifold as the unique local minimizer in its relative homology class among all nearby relative cycles?
\end{question}

We give an affirmative answer to Question~\ref{q2} for hypersurfaces, based on a mechanism different from White's, using varifold regularity theory. That leads to our first main result below. For brevity, we omit some notations. See Theorem \ref{main1precise} for the full statement.

\begin{theorem}[Theorem \ref{main1precise}]\label{main}
Let \(\Sigma\) be a compact, properly embedded, two-sided, strictly stable free-boundary minimal hypersurface. Then, in a sufficiently small neighborhood of \(\Sigma\), it is the unique mass minimizer in its relative homology class among relative cycles.
\end{theorem}

The preceding results describe the local variational picture in the index 0 case. We next turn to the case of positive Morse index. In finite-dimensional Morse theory, a critical point of index \(k\) is characterized locally as a \(k\)-dimensional saddle point. This suggests the following question:

\begin{question}
\label{q3}
Does a non-degenerate free-boundary minimal hypersurface of Morse index \(k\) satisfy a local \(k\)-parameter min-max characterization?
\end{question}

Strong local minimax properties for non--degenerate critical submanifolds in the boundaryless and fixed boundary setting were previously established by White \cite[Theorem 4]{White1994}. Our second main result gives an analogous result in free--boundary setting, affirmatively answers Question \ref{q3} for orientable hypersurfaces, based on a modification of our first main result. For brevity we don't give a full statement here, the precise version is given in Theorem \ref{main22}.
\begin{theorem}[Theorem \ref{main22}]\label{main2baby}
Let \(\Sigma\) be a compact, orientable, two-sided, properly embedded non--degenerate free-boundary minimal hypersurface of Morse index \(k>0\). Then there exists a local \(k\)-parameter family of relative cycles centered at \(\Sigma\), whose parameter boundary has mass strictly below \(\mathbf M(\Sigma)\), such that every filling of this boundary must reach the mass level \(\mathbf M(\Sigma)\).
\end{theorem}

We also strengthen Theorems \ref{main}, \ref{main2baby} from tubular to $\mathcal F_{rel}$-neighborhoods, see for Corollary \ref{cor: two enhance} and Theorem \ref{main2flat}. These strengthenings fit the natural workspace of Almgren--Pitts min--max theory. As an application, we obtain the existence of index 1 free--boundary minimal hypersurface:
\begin{theorem}[Theorem \ref{thm:index-one-omega-one}]
For compact Riemannian manifold $(N^{n+1},\partial N,g)$, where $g$ is chosen generically and $3\le n+1\le 7$, there exists a smooth, properly embedded, two-sided, multiplicity-one free-boundary minimal hypersurface \(\Gamma\) such that
\[
\mathbf M(\Gamma)=\omega_1(N,g)
\qquad\text{and}\qquad
\operatorname{Index}(\Gamma)=1.
\]

\end{theorem}

\subsection{Related Works and Motivation from Min--Max Theory}

The work most directly related to the present paper is White's fundamental work \cite{White1994}. In the strictly stable case, White proved that a smooth compact embedded critical submanifold is the unique local minimizer \cite[Theorem~2]{White1994}. In the presence of boundary, this is a fixed-boundary statement. In the same article, White also showed that a non--degenerate critical submanifold of Morse index \(k\) admits a local \(k\)-parameter min--max property \cite[Theorem~4]{White1994}. This result directly leads to the core ingredient in the study of the Morse index of min--max hypersurfaces \cite[Section~6]{MN2021}.

Several refinements of the index-0 statement are known. Inauen--Marchese proved a quantitative $\mathcal F$-neighborhood version of White's theorem \cite{IM2018}, while Morgan--Ros established similar result for strictly stable constant-mean-curvature hypersurfaces \cite{MR2010}.

The present paper develops parallel results for free-boundary minimal hypersurfaces. The essential distinction from the fixed-boundary case is that the boundary of a competitor is allowed to move along \(\partial N\), the natural competitors are therefore relative-cycles. For case of Morse index 0, we prove that a strictly stable hypersurface is the unique local mass minimizer in its relative homology class. For Morse index \(k\), we construct a \(k\)-parameter family and prove the corresponding local min--max characterization. Both results are first in a tubular neighborhood and then in a $\mathcal F_{rel}$--neighborhood.

The min--max motivation comes from the free-boundary min--max theory. The boundaryless theories was developed by Almgren and Pitts, with regularity due to Schoen--Simon, and has led to strong results concerning widths, multiplicity, and Morse index \cite{Almgren1965,Pitts1981,MN2016,Zhou2020,MN2021}. For fixed-boundary constructions, see De Lellis--Ramic \cite{DLR18}. In the free-boundary setting, the corresponding results were also well developed in \cite{FBMHI,FBMHII,FBMH1,WZC}. Free-boundary regularity and compactness results relevant here include \cite{Allard1972,GruterJost1986,ACS2018,GZ2018}. Since free-boundary sweepouts naturally fits the $\mathcal F_{rel}$ topology, the local theorems proved here provide tools for identifying the Morse index of multiplicity-one min--max realizations. As a first application, we obtain an index-one realization of the first free-boundary width.

\subsection{Idea of Proofs}

The proofs of Theorems \ref{main} and \ref{main2baby} share a common analytic core. We work directly in the relative-cycle setting and construct a free-boundary adapted tubular parametrization around \(\Sigma\). After directly choosing a minimizer in the neighborhood and showing that the minimizers converge back to $\Sigma$, both as currents and varifolds, followed by the Allard-type regularity theorems, every relevant minimizer are forced to be a free-boundary graph over \(\Sigma\).

In the strictly stable case, once a minimizer has been reduced to a  graph \(\Gamma_u\), Taylor expansion of the area functional forces \(\llbracket\Sigma\rrbracket\) to be the unique homologous minimizer, this proves the tubular-neighborhood local minimality theorem.

For a non--degenerate $\Sigma$ of Morse index \(k\), we use the spectral splitting \( H^1(\Sigma)=E_-^k\oplus E_+,\) where the second variation is negative on \(E_-^k\) and positive on \(E_+\). The negative eigenspace produces a canonical \(k\)-parameter family \(\{\gamma_a\}_{a\in\overline B^k}\) for which \(\Sigma=\gamma_0\) is the unique maximal slice. Given any competing \(k\)-parameter family with the same boundary values, we associate to each nearby graphical slice its unstable coordinates in \(E_-\). A degree argument yields a parameter at which all unstable coordinates vanish, thus the remaining displacement lies in the positive directions. Combining the graphical reduction and coercivity would give that all $k$--parameter families have maximal mass at least $\mathbf M(\Sigma)$. This gives the local \(k\)-parameter min--max characterization.

Finally, above two main results are upgraded to a $\mathcal F_{rel}$-neighborhood statement by replacement argument and free-boundary monotonicity formula. The corresponding Corollary \ref{cor: two enhance} acts as an avoidance principle in the free-boundary min--max cycle space. Combined with multiplicity one theorem, generic non--degeneracy, and the Morse-index upper bound, it yields an index-one realization of the first free-boundary width in Theorem \ref{thm:index-one-omega-one}.

\subsection{Organization of the Paper}

In Section~\ref{2}, we collect the basic notation and preliminary facts concerning relative \(\mathbb Z_2\)-cycles, varifolds, free-boundary variations, and strict stability.

In Section~\ref{3.1} we construct a free-boundary adapted tubular neighborhood of \(\Sigma\).

In Section~\ref{3}, we prove the main theorem. More precisely, in Section~\ref{3.2} we prove current and varifold convergence of the relative mass minimizers to \(\llbracket\Sigma\rrbracket\) and \(|\Sigma|\), respectively. In Section~\ref{3.3} we establish a uniform first variation bound. In Section~\ref{3.4} we apply regularity theorems to obtain graphical convergence. Finally, in Section~\ref{3.5} we use strict stability to prove the local minimizing property and complete the proof of the main theorem.

In Section \ref{proof of second result}, we use degree theory and a modified version of Theorem \ref{main} to characterize the local min--max property for a non-degenerate free--boundary minimal hypersurface of index $k.$

In Section~\ref{flatversion}, we prove a $\mathcal F_{rel}$-neighborhood modified local minimizing theorem by an exterior-replacement argument. We then apply this version to the first free-boundary width and prove an index-one conclusion for a multiplicity-one realization.

\subsection*{Acknowledgements}

The last author would like to thank Wenduo Zou for helpful discussions related to this work.

\section{Preliminaries}\label{2}

\subsection{Relative Currents and Relative Cycles}

We first recall some notation for relative currents and relative cycles. These are standard in geometric measure theory and in the free-boundary min--max literature with the notation slightly modified for brevity. See for example, Federer \cite{Federer1969}, Simon \cite{Simon1983} and Guang--Li--Wang--Zhou \cite[Section 3.1]{FBMHII}.

Throughout this subsection, the coefficient group \( \mathbb G \) should be taken in $\{\mathbb Z,\mathbb Z_2\}$, and all definitions below apply to either coefficient group. The index-0 results in Sections~\ref{3} and~\ref{flatversion} are formulated with coefficients in \(\mathbb Z_2\), whereas the index $k$ results are formulated with coefficients in \(\mathbb Z\). Orientations and signs must be retained under coefficient \(\mathbb Z\) and are irrelevant under coefficient \(\mathbb Z_2\).

Let \((K,A)\) be a compact pair. We denote by \(\mathcal R_k(K;\mathbb G)\) the space of \(k\)-dimensional rectifiable currents with coefficient $\mathbb G$ supported in \(K\). Define \(Z_k(K,A;\mathbb G):=\left\{T\in \mathcal R_k(K;\mathbb G):\operatorname{spt}(\partial T)\subset A\right\}.\) Thus an element of \(Z_k(K,A;\mathbb G)\) is a current whose boundary is allowed to lie in \(A\).

We say that two elements \(T,S\in Z_k(K,A;\mathbb G)\) are equivalent if \(T-S\in \mathcal R_k(A;\mathbb G),\) that is, if they differ by a current supported in \(A\). We denote the quotient space by
\[
\mathcal Z_k(K,A;\mathbb G):=
Z_k(K,A;\mathbb G)/\mathcal R_k(A;\mathbb G).
\]
Elements of \(\mathcal Z_k(K,A;\mathbb G)\) are called relative \(k\)-cycles in the pair \((K,A)\).

For \(\tau\in \mathcal Z_k(K,A;\mathbb G),\) we denote by \([\tau]\in H_k(K,A;\mathbb G)\) its relative homology class. Equivalently, if \(T,S\in Z_k(K,A;\mathbb G)\) are representatives, then \([T]=[S]\in H_k(K,A;\mathbb G)\) if and only if there exist \(Q\in \mathcal R_{k+1}(K;\mathbb G),\,R\in \mathcal R_k(A;\mathbb G),\) such that \(T-S=\partial Q+R.\)

For every \(\tau\in\mathcal Z_k(K,A;\mathbb G),\) there is a canonical representative \(T\in\tau\) satisfying \(T\llcorner A=0.\)  We shall often identify a relative cycle with its canonical representative. With this convention, if \(\Sigma^k\subset K\) is a properly embedded smooth submanifold with \(\partial\Sigma\subset A,\) then \(\llbracket\Sigma\rrbracket\in \mathcal Z_k(K,A;\mathbb G)\) denotes the relative cycle induced by \(\Sigma\). Notice that when \(\mathbb G=\mathbb Z\), we assume that \(\Sigma\) is orientable and fix an orientation for \(\Sigma\) and $\llbracket\Sigma\rrbracket$.

The mass of a relative cycle \(\tau\in\mathcal Z_k(K,A;\mathbb G)\) is defined by
\[
\mathbf M(\tau)
:=
\inf\{\mathbf M(T):T\in\tau\}.
\]
Equivalently, if \(T\) is the canonical representative of \(\tau\), then \( \mathbf M(\tau)=\mathbf M(T). \)

We shall also use the relative flat norm. For \( S\in \mathcal R_k(K;\mathbb G), \) define
\[
\mathcal F_{\mathrm{rel}}(S):=
\inf_{\substack{P\in \mathcal R_k(K;\mathbb G) \\
Q\in \mathcal R_{k+1}(K;\mathbb G) \\
R\in \mathcal R_k(A;\mathbb G)}}
\left\{
\mathbf M(P)+\mathbf M(Q):
S=P+\partial Q+R
\right\},
\]
Equivalently, for two relative cycles represented by \(T_1,T_2\), one may write
\[
\mathcal F_{\mathrm{rel}}(T_1-T_2)
=
\inf_{R\in\mathcal R_k(A;\mathbb G)}
\mathcal F(T_1-T_2+R).
\]

\subsection{Varifolds}

We next recall the notation of varifolds. Let \(N^{n+1}\) be a smooth Riemannian manifold. We denote by \(\mathcal G_n(N) \) the Grassmannian bundle of unoriented \(n\)-planes in \(TN\). An \(n\)-varifold \(V\in\mathcal V_n(N)\) is a Radon measure on \(\mathcal G_n(N)\). Its weight measure is denoted by \( \|V\|. \) Thus, for a Borel set \(B\subset N\),
\[
\|V\|(B):=V(\{(x,S)\in\mathcal  G_n(N):x\in B\}).
\]

If \(T\) is an \(n\)-dimensional rectifiable current, we denote by \( |T| \) the associated rectifiable varifold. Its weight measure agrees with the mass measure of \(T\), that is \(\| |T| \|(N)=\mathbf M(T). \) If \(\Sigma^n\subset N\) is a smooth embedded hypersurface, we write \( |\Sigma| \) for the varifold associated with \(\Sigma\), namely
\[
|\Sigma|(\varphi)=
\int_\Sigma \varphi(x,T_x\Sigma)\,d\mathcal H^n(x)
\]
for every \(\varphi\in C_c(\mathcal G_n(N))\).

We also introduce the varifold distance. Given $V,\,W\in\mathcal V_k(U)$, the varifold distance $\mathbf F_V$ between $V$ and $W$ is defined as \[ \mathbf F_V(V,W):=\sup_{\substack{ f\text{ Lipschitz}\\ \|f\|_\infty\le1\\ \operatorname{Lip}(f)\le1 }}\left\{\int_{\mathcal G_k(U)}f(x,S)dV(x,S)-\int_{\mathcal G_k(U)}f(x,S)dW(x,S)\right\}
  \]

\subsubsection*{First Variations}

For a \(C^1\) vector field \(X\) on \(N\), the first variation of an \(n\)-varifold \(V\) is defined by
\[
\delta V(X):=
\int_{G_n(N)} \operatorname{div}_S X(x)\,dV(x,S),
\]
where \( \operatorname{div}_S X(x) := \operatorname{tr}_S(DX(x)) \) is the tangential divergence of \(X\) along the \(n\)-plane \(S\). Equivalently, if \(e_1,\ldots,e_n \) is an orthonormal basis of \(S\), then \( \operatorname{div}_S X(x) = \sum_{i=1}^n \langle \nabla_{e_i}X,e_i\rangle. \)

We say that \(V\) has generalized mean curvature \(H\) in an open set \(U\) if
\[
H\in L^1_{\mathrm{loc}}(\|V\|\llcorner U)\qquad\text{and}\qquad\delta V(X)=
-\int_U \langle X,H\rangle\,d\|V\|
\]
for every compactly supported \(C^1\) vector field \(X\) in \(U\). In this paper we shall also use the corresponding free-boundary version, where the above identity is tested only against vector fields tangent to the ambient boundary, this will be recalled in the next subsection.

We shall frequently use the \(n\)-dimensional Jacobian of a linear map on an \(n\)-plane. Let \(A:E\to F \) be a linear map between finite-dimensional inner product spaces, and let \( S\subset E \) be an \(n\)-dimensional subspace. Choose an orthonormal basis \( e_1,\ldots,e_n \) of \(S\). We define
\[
J_n(A|_S):=
|Ae_1\wedge\cdots\wedge Ae_n|=
\sqrt{\det(\langle Ae_i,Ae_j\rangle)_{i,j=1}^n}.
\]
This definition is independent of the choice of orthonormal basis.

If \(f:U\to N'\) is differentiable at \(y\in U\), and if \( S\subset T_yU \) is an \(n\)-plane, we write
\[
J_nf(y,S):=
J_n(Df(y)|_S).
\]
Equivalently, if \(\xi\) is a unit simple \(n\)-vector spanning \(S\), then \( J_nf(y,S) = |\Lambda^n Df(y)(\xi)|. \)

These are standard notation in geometric measure theory; see for instance Federer \cite{Federer1969} and Simon \cite{Simon1983}.

\subsubsection*{Regularity Results}
The following theorem is a consequence of Allard's regularity theorem, restricted to the case of codimension 1 \cite{Allard1972}, see also \cite[Chapter 5]{Simon1983}.

\begin{theorem}
\label{thm:allard-interior}
Let \(n\in\mathbb N\), \(p>n\), and let \(\eta>0\). Then there exist constants \( \varepsilon(n,p,\eta),\ \gamma(n,p)>0, \) with the following property.

Let \(V=v(M,\theta)\) be a rectifiable \(n\)-varifold in \(B_\rho^{n+1}(0)\subset \mathbb R^{n+1}, \) with weight measure \( \mu=\|V\|. \) Assume that \( 0\in \operatorname{spt}\mu\), \( \theta\ge 1 \,\text{for }\mu\text{-a.e.} \) and that the density ratio satisfies \( {\mu(B_\rho(0))}/{\omega_n\rho^n}\le 1+\varepsilon. \) Assume moreover that \(V\) has generalized mean curvature \( H\in L^p(\mu) \) in \(B_\rho(0)\), and that
\[
\rho^{1-\frac np}
\left(
\int_{B_\rho(0)} |H|^p\,d\mu
\right)^{1/p}
\le \varepsilon.
\]
Then, after applying an orthogonal transformation $Q$ of \(\mathbb R^{n+1}\), there exists a function \( u:B_{\gamma\rho}^n(0)\to \mathbb R \) of class \(C^{1,\alpha}\), where \( \alpha=1-\frac np, \) such that \begin{enumerate}
\item $Du(0)=0$.
\item $\operatorname{spt}\mu\cap B_{\gamma\rho}(0)=Q(\operatorname{graph}u)\cap B_{\gamma\rho}(0).$
\item $\rho^{-1}\sup_{B_{\gamma\rho}^n(0)}|u| + \sup_{B_{\gamma\rho}^n(0)}|Du| + \rho^\alpha [Du]_{C^{0,\alpha}(B_{\gamma\rho}^n(0))} \le C\eta,$ where \(C=C(n,p)>0\).
\end{enumerate}
\end{theorem}

To adapt the case of free-boundary in this article, we shall use the analogue free-boundary version of Allard's regularity theorem. See \cite[Theorem 4.13]{GruterJost1986}.

\begin{theorem}\label{G}
Let \(n\in\mathbb N\), \(p>n\), and let \(\eta>0\). Then there exist constants \( \gamma(n,p),\,\varepsilon(n,p,\eta)>0 \) with the following property.

Let \(\Gamma\subset B_\rho^{n+1}(0)\subset \mathbb R^{n+1}\) be a \(C^2\) hypersurface with \(0\in\Gamma\). Let \(\Omega\) be one of the two sides of \(\Gamma\). Assume that the curvature of \(\Gamma\) satisfies \(\rho \sup_{\Gamma\cap B_\rho(0)} |A_\Gamma| \le \varepsilon^2. \)

Let \(V=v(M,\theta)\) be a rectifiable \(n\)-varifold in \(\overline \Omega\cap B_\rho^{n+1}(0)\), and denote \( \mu=\|V\|. \) Assume that \( 0\in \operatorname{spt}\mu\), \(\theta^n(\mu,x)\ge 1 \text{ for }\mu\text{-a.e.}\) and that \({\mu(B_\rho(0))}/{\omega_n\rho^n} \le \frac12(1+\varepsilon). \) Assume that \(V\) has free-boundary generalized mean curvature \(H\in L^p(\mu)\) in \(B_\rho(0)\), and that
\[
\rho^{1-\frac np}
\left(
\int_{B_\rho(0)} |H|^p\,d\mu
\right)^{1/p}
\le \varepsilon.
\]
Then, after applying an orthogonal transformation $Q$ of \(\mathbb R^{n+1}\), there exists a function \( u:B_{\gamma\rho}^n(0)\to \mathbb R \) of class \(C^{1,\alpha}\), where \( \alpha=\min\{\frac12,1-\frac np\}, \) such that \begin{enumerate}
\item $u(0)=0$.
\item $\operatorname{spt}\mu\cap B_{\gamma\rho}(0)=Q(\operatorname{graph}u)\cap B_{\gamma\rho}(0)\cap \overline\Omega$.
\item $\rho^{-1}\sup_{D_{\gamma\rho}} |u|+\sup_{D_{\gamma\rho}} |Du|+\rho^\alpha[Du]_{C^{0,\alpha}(D_{\gamma\rho}^n)}\le C\eta,$ where \(C=C(n,p)>0\) and $D_{\gamma\rho}=\operatorname{proj}_{\mathbb R^n\times \mathbb R \to \mathbb R^n}\left(\operatorname{graph}u\cap Q^{-1}\bigl(B_{\gamma\rho}^{n+1}(0)\cap \overline{\Omega}\bigr)\right).$
\item $\nu_\Gamma(x)\in T_xQ(\operatorname{graph}u)$ for $x\in Q(\operatorname{graph}u)\cap\Gamma\cap B_{\gamma\rho}(0).$
\end{enumerate}
\end{theorem}
\begin{remark}
\label{rem:Riemannian-regularity}
We shall use Theorems~\ref{thm:allard-interior} and~\ref{G} in their local Riemannian forms. At an interior point we work in geodesic normal coordinates, while at a boundary point we use Fermi coordinates. After rescaling a coordinate ball of radius \(R\) to unit size, the pulled-back metrics converge uniformly in \(C^2\) to the Euclidean metric as \(R\to0\). Consequently, after decreasing the smallness thresholds if necessary, the constants in the corresponding regularity theorems may be chosen uniformly on sufficiently small coordinate balls.

Accordingly, in the proof of Proposition \ref{convergence as graphs} we verify the mass-ratio and mean curvature conditions intrinsically in \((N,\partial N,g)\).\end{remark}

\subsection{Free-Boundary Variations}
The following notions are standard in the free-boundary setting. See for example \cite{GruterJost1986}\cite[Section 2.2]{FBMHI}\cite[Section 2]{ACS2018}.

\subsubsection*{Free--Boundary First Variations}
We now recall the class of admissible vector fields in the free-boundary setting. Let \((N^{n+1},\partial N,g)\) be a smooth Riemannian manifold with boundary. We denote by \( \mathfrak X^{\partial}(N) \) the space of smooth vector fields \(X\) on \(N\) satisfying
\[
X(p)\in T_p\partial N
\qquad
\text{for every }p\in\partial N.
\]
Equivalently, if \(\mu_{\partial N}\) denotes the unit normal to \(\partial N\) in \(N\), then \( \langle X,\mu_{\partial N}\rangle=0 \text{ on }\partial N. \) Note that if \(X\in\mathfrak X^{\partial}(N)\), then its local flow \(\psi_t\) preserves the ambient boundary, i.e. \( \psi_t(\partial N)\subset \partial N \) for all sufficiently small \(t\). This is the natural class of variations for free-boundary problems. If \(\Sigma^n\subset N\) is a properly embedded hypersurface with \( \partial\Sigma\subset\partial N, \) then the flow of any \(X\in\mathfrak X^{\partial}(N)\) sends \(\Sigma\) to a family of hypersurfaces whose boundaries remain on \(\partial N\), i.e. \( \partial(\psi_t(\Sigma))\subset\partial N. \)

An \(n\)-varifold \(V\) in \(N\) is said to have free-boundary first variation with generalized mean curvature \(H\) if
\[
\delta V(X)
=
-\int_N \langle X,H\rangle\,d\|V\|
\]
for every compactly supported vector field \( X\in\mathfrak X^{\partial}(N). \) In particular, \(V\) is called free-boundary stationary if \( \delta V(X)=0 \) for every compactly supported \( X\in\mathfrak X^{\partial}(N). \)

For a smooth properly embedded hypersurface \(\Sigma\subset N\), the free-boundary stationarity condition is equivalent to
\[
H_\Sigma=0
\quad\text{in }\Sigma\qquad\text{and}\qquad\Sigma\perp\partial N
\quad\text{along }\partial\Sigma.
\]
Indeed, if \(\eta\) denotes the outward conormal of \(\partial\Sigma\) in \(\Sigma\), then the first variation formula gives
\[
\delta\Sigma(X)=
-\int_\Sigma \langle H_\Sigma,X\rangle\,d\mathcal H^n+
\int_{\partial\Sigma}\langle X,\eta\rangle\,d\mathcal H^{n-1},
\]
the stationary condition then follows immediately.

\subsubsection*{Second Variations}
\label{subsec:second-variation}

Let \( \Sigma^n\subset (N^{n+1},\partial N,g) \) be a compact, two-sided, properly embedded free-boundary minimal hypersurface, and let \(\nu\) be a global unit normal vector field along \(\Sigma\). We denote by \(A_\Sigma\) the second fundamental form of \(\Sigma\), and by \(h^{\partial N}\) the second fundamental form of \(\partial N\) in \(N\).

For \(u,v\in H^1(\Sigma)\), the free-boundary second-variation bilinear form is
\[
Q(u,v)=
\int_\Sigma
\left(
\langle\nabla u,\nabla v\rangle-
\bigl(
|A_\Sigma|^2+\operatorname{Ric}_N(\nu,\nu)
\bigr)uv
\right)
\,d\mathcal H^n-
\int_{\partial\Sigma}
h^{\partial N}(\nu,\nu)\,uv
\,d\mathcal H^{n-1},
\]
where the boundary term is well defined for \(H^1\)-functions by the trace theorem. The associated Jacobi operator is
\[
L_\Sigma u=
\Delta_\Sigma u+
\bigl(
|A_\Sigma|^2+\operatorname{Ric}_N(\nu,\nu)
\bigr)u
\]
together with its natural free-boundary Robin boundary condition. We use the weak eigenvalue formulation: a number \(\lambda\in\mathbb R\) is an eigenvalue if there exists a nonzero function \(\varphi\in H^1(\Sigma)\) such that
\[
Q(\varphi,w)=
\lambda\int_\Sigma \varphi w\,d\mathcal H^n
\qquad
\text{for every }w\in H^1(\Sigma).
\]
By elliptic regularity, every eigenfunction is smooth up to \(\partial\Sigma\).

The spectrum is discrete and may be written, counting multiplicity, as \( \lambda_1\le\lambda_2\le\cdots\longrightarrow+\infty. \) Let \( \{\varphi_i\}_{i=1}^{\infty} \) be a corresponding \(L^2(\Sigma)\)-orthonormal basis of eigenfunctions.

The Morse index of \(\Sigma\) is the maximal dimension of a subspace of \(H^1(\Sigma)\) on which \(Q\) is negative definite, and the nullity of \(\Sigma\) is the dimension of \(\ker Q\). Equivalently,
\[
\operatorname{index}(\Sigma)=
\#\{i:\lambda_i<0\},
\qquad
\operatorname{nullity}(\Sigma)=
\#\{i:\lambda_i=0\},
\]
where eigenvalues are counted with multiplicity. Define
\[
E_-:=
\operatorname{span}\{\varphi_i:\lambda_i<0\}
\qquad
E_0:=
\operatorname{span}\{\varphi_i:\lambda_i=0\}
\qquad E_+:=
\overline{
\operatorname{span}\{\varphi_i:\lambda_i>0\}
}^{\,H^1(\Sigma)},
\]
Then \( H^1(\Sigma) = E_-\oplus E_0\oplus E_+. \) This decomposition is orthogonal in \(L^2(\Sigma)\), and the summands are also mutually \(Q\)-orthogonal. Moreover, \(Q\) is negative definite on \(E_-\), vanishes on \(E_0\), and is positive definite on \(E_+\). For similar decomposition and its property, see \cite[Proposition-Definition 2.23]{CMC23}.

In particular, if \(\operatorname{index}(\Sigma)=k,\) and \( \operatorname{nullity}(\Sigma)=0, \) then \( \lambda_k<0<\lambda_{k+1}, \) and
\[
E_-=\operatorname{span}\{\varphi_1,\ldots,\varphi_k\},
\qquad
H^1(\Sigma)=E_-\oplus E_+.
\]
Since \(E_-\) is finite dimensional and the spectrum has a gap at zero, together with G{\aa}rding inequality, we imply that there exists \(c_0>0\) such that
\[
Q(v,v)\le
-c_0\|v\|_{H^1(\Sigma)}^2
\qquad
Q(w,w)\ge
c_0\|w\|_{H^1(\Sigma)}^2
\qquad 
\text{for every }v\in E_-,\,w\in E_+.
\]

We say that \(\Sigma\) is stable if \( Q(u,u)\ge0\) for every \(u\in H^1(\Sigma),\) or equivalently if \(\lambda_1\ge0\). We say that \(\Sigma\) is strictly stable if \( Q(u,u)>0\) for every non-zero \(u\in H^1(\Sigma), \) or equivalently if \( \lambda_1>0. \) Thus strict stability is precisely the special case \( \operatorname{index}(\Sigma)=0,\, \operatorname{nullity}(\Sigma)=0. \)

Finally, if ambient metric is bumpy, then every free-boundary minimal hypersurface is non-degenerate, consequently every stable free-boundary minimal hypersurface is automatically strictly stable.

\section{Construction of Tubular Neighborhood}\label{3.1}
For the rest of this paper, we work with a compact smooth Riemannian manifold with boundary $(N^{n+1},g)$ and $\Sigma^n\subset N$ be a compact, two-sided, properly embedded free-boundary minimal hypersurface, thus $\partial \Sigma\subset \partial N$ and \(\Sigma\) meets \(\partial N\) orthogonally along \(\partial\Sigma\). Here properly embedded means that \( \Sigma\cap\partial N=\partial\Sigma. \)

An important prework is to choose a tubular neighborhood which is compatible with the ambient boundary. Foliated neighborhoods of non-degenerate free-boundary minimal hypersurfaces were also constructed by Wang~\cite[Lemma 2.5, 2.6]{WZC}. In particular, Wang obtained free-boundary foliations with mean curvature of fixed sign on the nearby leaves. In this paper we use a slightly different adapted tubular parametrization, based on a defining function whose gradient is tangent to \(\partial N\). This parametrization is chosen so that the associated projection and retraction preserve the relative pair and admit the Jacobian estimates needed for the later argument.

\begin{lemma}\label{Construction of NBHD}

There exist an open neighborhood \(\mathcal U\) of \(\Sigma\) in \(N\) and a smooth function \(\tau:\mathcal U\to \mathbb R \) such that\begin{enumerate}
\item $\Sigma=\tau^{-1}(0)$.
\item $d\tau\neq 0$ on $ \mathcal U.$
\item $\partial_{\mu_{\partial N}}\tau=0$ on $\mathcal U\cap \partial N,$ where \(\mu_{\partial N}\) denotes the unit inward normal to \(\partial N\) in \(N\). Equivalently, \( \nabla \tau \in T\partial N\) on \(\mathcal U\cap \partial N\).
\end{enumerate}
\end{lemma}

\begin{proof}

We first construct \(\tau\) locally and then patch the local functions.

Let \(p\in \partial\Sigma\). Since \(\Sigma\) is properly embedded and meets \(\partial N\) orthogonally, we may choose local coordinates \( (z_1,\ldots,z_{n-1},y,u) \) near \(p\), where
\[
\partial N=\{u=0\},
\quad
N=\{u\ge 0\},\quad\text{and}\quad\partial\Sigma=\{y=0,\ u=0\}.
\]
Here \(z=(z_1,\ldots,z_{n-1})\) are coordinates along \(\partial\Sigma\), \(y\) is the direction in \(\partial N\) normal to \(\partial\Sigma\), and \(u\) is the boundary collar coordinate. By the implicit function theorem, after shrinking the coordinate neighborhood (denote $U_p$) if necessary, \(\Sigma\) can be written as \( \Sigma=\{y=f(z,u)\} \) for a smooth function \(f\). Since \(\partial\Sigma\subset \partial N\) is given by \(y=0\) when \(u=0\), we have \(f(z,0)=0. \) The free-boundary orthogonality condition implies \(\partial_u f(z,0)=0. \) Therefore the local function \( \tau_p(z,y,u):=y-f(z,u) \) satisfies
\[
\tau_p=0 \quad \text{on } \Sigma\qquad\text{and}\qquad\tau_p \neq 0 \text{ by } \partial_y \tau_p=1.
\]
Furthermore,
\[
\partial_u \tau_p(z,y,0)=
-\partial_u f(z,0)=0.
\]
Thus \( \partial_{\mu_{\partial N}}\tau_p=0\) on \(\partial N \) in this coordinate patch.

For points \(p\in \Sigma\setminus\partial\Sigma\), we choose ordinary local coordinates in the interior of \(N\) and take a local defining function \(\tau_p\) for \(\Sigma\). Since \(\Sigma\) is two-sided, the local defining functions can be chosen with consistent sign, i.e. $d\tau_p(\nu_\Sigma) > 0$ on $U_p\cap \Sigma$, where $\nu_\Sigma$ is the global unit normal on $\Sigma$.

Take a finite open cover \(\{U_\alpha\}\) of \(\Sigma\) by such coordinate neighborhoods. Choose a smooth partition of unity \(\{\chi_\alpha\}\) subordinate to \(\{U_\alpha\}\), with $\partial_{\mu_{\partial N}}\chi_\alpha=0$ on $\partial N.$ This is obtained by first taking a partition of unity on \(\partial N\), extending it constantly along the boundary collar direction, and then adding interior cutoff functions supported away from \(\partial N\). Define \( \tau:=\sum_\alpha \chi_\alpha \tau_\alpha. \) Since each \(\tau_\alpha\) and the terms \(d\chi_\alpha\,\tau_\alpha\) vanishes on \(\Sigma\), we have
\[
\tau=0 \quad\text{  and  }\quad d\tau=
\sum_\alpha \chi_\alpha d\tau_\alpha
\quad
\mbox{on } \Sigma.
\]
By $d\tau_\alpha(\nu_\Sigma) > 0$ for each $\alpha$, $d\tau$ is nonzero on \(\Sigma\). Shrinking \(\mathcal U\) if necessary, we may assume \( d\tau\neq 0\) on \(\mathcal U. \)

Finally, on \(\mathcal U\cap\partial N\), each local defining function satisfies \( \partial_{\mu_{\partial N}}\tau_\alpha=0, \) and the partition functions satisfy \( \partial_{\mu_{\partial N}}\chi_\alpha=0. \) Therefore
\[
\partial_{\mu_{\partial N}}\tau=
\sum_\alpha
(\partial_{\mu_{\partial N}}\chi_\alpha)\tau_\alpha+
\sum_\alpha
\chi_\alpha(\partial_{\mu_{\partial N}}\tau_\alpha)=0
\quad
\mbox{on } \mathcal U\cap\partial N.
\]
This proves the lemma.
\end{proof}

We now define a transverse vector field by \( Z:=\frac{\nabla \tau}{|\nabla \tau|^2}. \) Here \(\nabla\tau\) is the gradient with respect to the ambient metric \(g\). Since \(d\tau\neq0\) on \(U\), this vector field is smooth. Moreover,
\[
d\tau(Z)=
g\left(\nabla\tau,\frac{\nabla\tau}{|\nabla\tau|^2}\right)=1.
\]
Because \(\nabla\tau\in T\partial N\) on \(U\cap\partial N\), we also have \( Z\in T\partial N \) on $U\cap\partial N.$ Thus the flow of \(Z\) preserves \(\partial N\). Let \(\Phi_s\) denote the associated local flow of \(Z\). Define
\[
F:\Sigma\times (-r_0,r_0)\to N \qquad \text{by }F(x,s):=\Phi_s(x).
\]
After shrinking \(r_0>0\) if necessary, the map \(F\) is a diffeomorphism onto its image. Indeed, \( F(x,0)=x, \) and \( \partial_s F(x,0)=Z(x). \) Since \(\tau=0\) on \(\Sigma\) and \(d\tau(Z)=1\), the vector \(Z(x)\) is transverse to \(T_x\Sigma\). Hence
\[
D F_{(x,0)}:T_x\Sigma\oplus \mathbb R\to T_xN
\]
is an isomorphism. The inverse function theorem gives the local result, and compactness of \(\Sigma\) gives a uniform \(r_0\). Furthermore, we also have $F(\partial\Sigma\times(-r_0,r_0))\subset \partial N$ since \(Z\) is tangent to \(\partial N\) and its flow also preserves \(\partial N\).

For \(x\in\Sigma\), we compute
\[
\frac{d}{ds}\tau(F(x,s))=
d\tau(\partial_sF(x,s))=
d\tau(Z(F(x,s)))=1.
\]
Since \(\tau(F(x,0))=0\), it follows that \(\tau(F(x,s))=s\) is exactly the height function of $F$. Therefore
\[
F(\Sigma\times\{s\})=\{\tau=s\}\cap F(\Sigma\times(-r_0,r_0)).
\]
We denote \(\Sigma_s:=F(\Sigma\times\{s\})=\{\tau=s\}\cap \mathcal U.\) Since \(T\Sigma_s=\ker d\tau,\) and \(Z=\frac{\nabla\tau}{|\nabla\tau|^2},\) we have \(Z\perp T\Sigma_s.\) Consequently the foliation \(\{\Sigma_s\}_{|s|<r_0}\) is transverse and orthogonal to the flow direction \(Z\). For \(0<r<r_0\), define
\[
K_r:=F(\Sigma\times[-r,r])\qquad A_r:=K_r\cap \partial N.
\]
Since the flow preserves \(\partial N\), we have \(A_r=F(\partial\Sigma\times[-r,r]).\) We also define the tubular projection
\[
\pi_r:K_r\to \Sigma \qquad\text{by }\pi_r(F(x,s))=x.
\]
See Figure~\ref{fig:tubular-retraction} for an illustration:

\begin{figure}[H]
\centering
\begin{tikzpicture}[x=1cm,y=1cm]

  \begin{scope}
    \path[fill=fbblue!9]
      plot[domain=0:3,samples=50]
        ({1.15+0.10*\x*\x-0.45/sqrt(1+0.04*\x*\x)},
         {\x+0.09*\x/sqrt(1+0.04*\x*\x)})
      -- plot[domain=3:0,samples=50]
        ({1.15+0.10*\x*\x+0.45/sqrt(1+0.04*\x*\x)},
         {\x-0.09*\x/sqrt(1+0.04*\x*\x)}) -- cycle;

    \draw[line width=.9pt] (-0.05,0)--(3.05,0)
      node[below right] {$\partial N$};
    \draw[auxiliary]
      plot[domain=0:3,samples=50]
        ({1.15+0.10*\x*\x-0.45/sqrt(1+0.04*\x*\x)},
         {\x+0.09*\x/sqrt(1+0.04*\x*\x)});
    \draw[maincurve]
      plot[domain=0:3,samples=40] ({1.15+0.10*\x*\x},\x);
    \draw[auxiliary]
      plot[domain=0:3,samples=50]
        ({1.15+0.10*\x*\x+0.45/sqrt(1+0.04*\x*\x)},
         {\x-0.09*\x/sqrt(1+0.04*\x*\x)});
    \node[anchor=east] at (.98,2.60) {$\Sigma_{-r}$};
    \draw[line width=.45pt] (1.02,2.60)--(1.30,2.60);
    \node[anchor=west] at (2.22,2.84) {$\Sigma$};
    \draw[line width=.45pt] (1.96,2.84)--(2.18,2.84);
    \node[anchor=west] at (2.42,1.84) {$\Sigma_r$};
    \draw[line width=.45pt] (1.98,1.84)--(2.38,1.84);

    \node[fbblue,anchor=west] at (1.90,.82) {$K_r$};
    \fill[fbblue] (1.15,0) circle (1.35pt);
    \node[below=4pt] at (1.15,0) {$\partial\Sigma$};

    \draw[line width=.55pt] (1.15,0.18)--(1.33,0.18)--(1.33,0);

    \coordinate (q1) at
      ({1.15+0.10*1.25*1.25+0.45/sqrt(1+0.04*1.25*1.25)},
       {1.25-0.09*1.25/sqrt(1+0.04*1.25*1.25)});
    \coordinate (p1) at ({1.15+0.10*1.25*1.25},1.25);
    \draw[maparrow] (q1)--(p1)
      node[midway,above] {$\pi_r$};
    \coordinate (q2) at
      ({1.15+0.10*2.05*2.05-0.45/sqrt(1+0.04*2.05*2.05)},
       {2.05+0.09*2.05/sqrt(1+0.04*2.05*2.05)});
    \coordinate (p2) at ({1.15+0.10*2.05*2.05},2.05);
    \draw[maparrow] (q2)--(p2);

  \end{scope}
\end{tikzpicture}
\caption{The free-boundary adapted tubular neighborhood.}
\label{fig:tubular-retraction}
\end{figure}

Since $Z$ is normal to $\Sigma$, we may denote
\begin{equation}\label{eq: zeta}
Z|_\Sigma=\zeta\,\nu_\Sigma,
\qquad
\zeta:=\langle Z,\nu_\Sigma\rangle
=\frac{1}{|\nabla\tau|}>0.
\end{equation}
Consequently, if $u\in C^\infty(\Sigma)$ and \(\Gamma_{tu}:=\{F(x,tu(x)):x\in\Sigma\},\) then the normal velocity of the variation $t\mapsto\Gamma_{tu}$ at $t=0$ is \(uZ|_\Sigma=\zeta u\,\nu_\Sigma.\) Thus the graph height $u$ with respect to the tubular parametrization corresponds to the scalar normal variation $\zeta u$.

\begin{lemma}\label{Jestimate}
Under the free-boundary tubular parametrization constructed above, after decreasing \(r_0>0\) if necessary, there exists two positive constants \(C\) and $c$, independent of \(r\), such that for every \(0<r<r_0\), every \( y\in K_r, \) and every \(n\)-plane $S\subset T_yN$, one has\begin{enumerate}
\item \(J_n\pi_r(y,S)\le 1+Cr\).
\item $J_n\pi_r(y,S)\le 1+Cr-c\,\operatorname{dist}_{\mathcal G}\left(S,T_{\pi_r(y)}\Sigma\right)^2.$
\end{enumerate}
Where \( J_n\pi_r(y,S) \) denotes the \(n\)-dimensional Jacobian of \(D\pi_r(y)\) restricted to \(S\) and $\operatorname{dist}_{\mathcal G}$ denotes the grassmannian distance that we would explain specifically later in the proof.
\end{lemma}

\begin{proof}
Write
\[
y=F(x,t),
\qquad
x\in\Sigma,
\qquad
|t|\le r.
\]
We first prove that \(D\pi_r(y)\) is uniformly close to the orthogonal projection onto \( T_x\Sigma. \)

Let
\[
\mathcal P_{y\rightarrow x}:T_yN\to T_xN
\]
denote parallel transport along the curve
\[
s\mapsto F(x,(1-s)t),
\qquad 0\le s\le1.
\]
Define
\[
\mathcal A_{x,t}:=D\pi_r(F(x,t))\circ \mathcal P_{y\rightarrow x}^{-1}:T_xN\to T_x\Sigma.
\]
Since \(F\) and \(\pi_r\) are smooth and \(\Sigma\) is compact, the map \( (x,t)\mapsto \mathcal A_{x,t} \) is smooth for \(|t|<r_0\). At \(t=0\), we have \(F(x,0)=x\), and by definition \( \pi_r(F(x,t))=x. \) Hence
\[
D\pi_r(x)|_{T_x\Sigma}=\operatorname{id}_{T_x\Sigma},
\qquad
D\pi_r(x)(\nu_\Sigma(x))=0.
\]
Therefore \( \mathcal A_{x,0}=\operatorname{pr}_{T_x\Sigma}, \) where \(\operatorname{pr}_{T_x\Sigma}:T_xN\to T_x\Sigma\) is the orthogonal projection.

By smoothness and compactness, there is a constant \(C_1>0\) such that
\[
\|\mathcal A_{x,t}-\operatorname{pr}_{T_x\Sigma}\|\le C_1|t|
\]
for all \(x\in\Sigma\) and all \(|t|<r_0\). In particular, since \(|t|\le r\), \( \|\mathcal A_{x,t}-\operatorname{pr}_{T_x\Sigma}\|\le C_1r. \)

Let \(S\subset T_yN\) be an \(n\)-plane and set \( \widetilde S:=\mathcal P_{y\rightarrow x}(S)\subset T_xN. \) Since parallel transport is an isometry, \( J_n\pi_r(y,S) = J_n(\mathcal A_{x,t},\widetilde S). \) Let \(\xi\) be a unit simple \(n\)-vector spanning \(\widetilde S\). Then
\[
J_n(\mathcal A_{x,t},\widetilde S)=
|\Lambda^n \mathcal A_{x,t}(\xi)|.
\]
The map \( \mathcal A\mapsto \Lambda^n \mathcal A \) is locally Lipschitz on bounded subsets of \(\operatorname{Hom}(T_xN,T_x\Sigma)\). Since \(\mathcal A_{x,t}\) and \(\operatorname{pr}_{T_x\Sigma}\) are uniformly bounded, there exists \(C_2>0\) such that
\[
|\Lambda^n \mathcal A_{x,t}(\xi)|\le
|\Lambda^n \operatorname{pr}_{T_x\Sigma}(\xi)|+
C_2\|\mathcal A_{x,t}-\operatorname{pr}_{T_x\Sigma}\|.
\]
Thus
\[
J_n\pi_r(y,S)\le
J_n(\operatorname{pr}_{T_x\Sigma},\widetilde S)+C_2C_1r.
\]
Since \(\operatorname{pr}_{T_x\Sigma}\) is an orthogonal projection, it is \(1\)-Lipschitz, and hence \( J_n(\operatorname{pr}_{T_x\Sigma},\widetilde S)\le1. \) Therefore we have \( J_n\pi_r(y,S)\le 1+Cr \) for some constant \(C>0\), independent of \(r\). This concludes 1.

To conclude 2, we first introduce the notation of grassmannian distance:\[ \operatorname{dist}_{\mathcal G} \left( S, T_x\Sigma \right) := \left\| \operatorname{pr}_{\widetilde S} - \operatorname{pr}_{T_x\Sigma} \right\|_{\operatorname{HS}},
\]
here \(\operatorname{pr}_{\widetilde S_y}\) and \(\operatorname{pr}_{T_x\Sigma}\) are the orthogonal projections onto \(\widetilde S\) and \( T_x\Sigma\) inside \(T_xN\), and \(\|\cdot\|_{\operatorname{HS}}\) denotes the Hilbert-Schmidt norm.

Let us introduce an elementary linear algebra estimate\[ J_n(\operatorname{pr}_{T_x\Sigma},\widetilde S)\le 1-c_0\operatorname{dist}_{\mathcal G}\left(\widetilde S,T_x\Sigma\right)^2
\]
and briefly justify it. Since both \(\widetilde S\) and \(T_x\Sigma\) are \(n\)-planes in the \((n+1)\)-dimensional inner product space \(T_xN\), there is only one non-zero principal angle between them. Let this angle be \(\theta\in[0,\pi/2]\). Then \( J_n(\operatorname{pr}_{T_x\Sigma},\widetilde S) = \cos\theta. \) Moreover,
\[
\operatorname{dist}_{\mathcal G}
\left(
\widetilde S,T_x\Sigma
\right)^2
\asymp
\sin^2\theta.
\]
Since \( 1-\cos\theta = \frac{\sin^2\theta}{1+\cos\theta} \ge \frac12\sin^2\theta, \) we get \( \cos\theta \le 1-c_0 \operatorname{dist}_{\mathcal G} \left( \widetilde S,T_x\Sigma \right)^2 \) for some constant \(c_0>0\). This proves the linear algebra estimate.

Combining the previous estimates gives
\[
J_n\pi_r(y,S)\le
1+Cr-c_0
\operatorname{dist}_{\mathcal G}
\left(
\widetilde S,T_x\Sigma
\right)^2.
\]
By the definition of the Grassmannian distance between \(S\subset T_yN\) and \(T_x\Sigma\subset T_xN\), this is exactly
\[
J_n\pi_r(y,S)\le
1+Cr-c
\operatorname{dist}_{\mathcal G}
\left(
S,T_{\pi_r(y)}\Sigma
\right)^2
\]
as desired.
\end{proof}

Finally, for later use, define the retraction $P_r:F(\Sigma\times(-r_0,r_0))\to K_r$ by $P_r(F(x,s)):=F(x,q_r(s)),$ where
\[
q_r(s):=
\begin{cases}
-r, & s<-r,\\
s, & -r\le s\le r,\\
r, & s>r.
\end{cases}
\]
Then
\[
P_r|_{K_r}=\operatorname{id}_{K_r},\qquad P_r(\partial N\cap \mathcal U)\subset \partial N.
\]
Although \(P_r\) is generally not smooth along the hypersurfaces \(\{s=\pm r\}\), it is Lipschitz. Thus it induces push-forwards on currents and varifolds. Also note that $q_r$ is 1-Lipschitz and consequently \(\operatorname{Lip}(P_r)\leq C,\) where $C=C(F)$ is independent of $r.$

Moreover, by Rademacher's theorem, for every countably \(n\)-rectifiable set \(M\subset \mathcal U\), the restriction \(P_r|_M\) is approximately tangentially differentiable at \(\mathcal H^n\)-almost every point of \(M\). At such a point \(z\), we shall use the notation \( D P_r(z)\big|_{\operatorname{Tan}(M,z)} \) for the approximate tangential differential \( \operatorname{ap}D(P_r|_M)(z). \) The same convention will be used for compositions of \(P_r\) with smooth maps.


\section{Proof of Theorem \ref{main}}\label{3}

In this section, we assume that \(\Sigma\) is strictly stable for area functional, and denote $\llbracket \Sigma \rrbracket$ the mod 2 current induced by $\Sigma$.

We will divide the full proof of Theorem~\ref{main} into several subsections. First, we restate the theorem with notations in Section \ref{3.1}:
\begin{theorem}\label{main1precise}
Let \((N^{n+1},g)\) be a compact smooth Riemannian manifold with boundary. Let \(\Sigma^n\subset N \) be a compact, two-sided, properly embedded free-boundary minimal hypersurface. Assume that \(\Sigma\) is strictly stable for area functional, then there exist \(r_0>0\) and a free-boundary adapted tubular parametrization
\[
F:\Sigma\times (-r_0,r_0)\to N\qquad\text{with }F(x,0)=x,
\qquad
F(\partial\Sigma\times (-r_0,r_0))\subset \partial N,
\]
such that the following holds: for \(0<r<r_0\), denote \( K_r:=F(\Sigma\times[-r,r]),\) \( A_r:=K_r\cap \partial N. \) If \(T\in \mathcal Z_n(K_r,A_r;\mathbb Z_2) \) satisfies \( [T]=[\Sigma]\in H_n(K_r,A_r;\mathbb Z_2), \) then
\[
\mathbf M(T)\ge \mathbf M(\Sigma).
\]
Moreover, equality holds if and only if \( T=\llbracket\Sigma\rrbracket \) as a relative \(\mathbb Z_2\)-cycle in \((K_r,A_r)\).
\end{theorem}

\begin{remark}
  If \(\Sigma\) is orientable, the same theorem holds with coefficient \(\mathbb Z\). 
\end{remark}

For each small \(r\), we first choose a relative mass minimizer in the class of \(\llbracket\Sigma\rrbracket\) in each previous neighborhood. The main task is to show that these minimizers converge back to \(\Sigma\) with multiplicity one and, after applying regularity theorems, are actually graphs over \(\Sigma\). The strict stability of \(\Sigma\) then rules out every nontrivial graph.

\subsection{Convergence as \texorpdfstring{$\mathbb{Z}_2$}{Z2}-Currents and Varifolds}\label{3.2}
The goal of this subsection is to prove that these minimizers converge back to \(\llbracket\Sigma\rrbracket\) as relative currents and to \(|\Sigma|\) as varifolds. We consider the class
\[
\mathcal C_r:=
\left\{
T\in \mathcal Z_n(K_r,A_r;\mathbb Z_2):
[T]=[\Sigma]\ \text{in } H_n(K_r,A_r;\mathbb Z_2)
\right\}.
\]
Define \( m_r:=\inf\{\mathbf M(T):T\in\mathcal C_r\}. \) By compactness of relative currents modulo \(2\) \cite[Theorem 2.3]{WEYLLAW}, there exists \( T_r\in\mathcal C_r \) such that \( \mathbf M(T_r)=m_r. \) Equivalently,
\[
\mathbf M(T_r)=
\inf
\left\{
\mathbf M(T):
T\in \mathcal Z_n(K_r,A_r;\mathbb Z_2),
\ [T]=[\Sigma]\in H_n(K_r,A_r;\mathbb Z_2)
\right\}.
\]
We call \(T_r\) a mass-minimizing relative mod \(2\) representative of \([\Sigma]\) in \((K_r,A_r)\), and will show that:\begin{enumerate}
\item $T_r\rightarrow \llbracket \Sigma \rrbracket$ as mod 2 currents, where $\llbracket \Sigma \rrbracket$ is the mod 2 current induced by $\Sigma$.
\item $V_r:=|T_r|\rightarrow |\Sigma|$ as varifolds, where $|\Sigma|$ is the varifold induced by $\Sigma$. Moreover, the varifold convergence has multiplicity 1.
\end{enumerate}

The key point is that the relative homology constraint is strong enough to control the projection of \(T_r\) onto \(\Sigma\). This will give the lower mass bound needed to force convergence. First we show that $(\pi_r)_{\#} T_r=\llbracket \Sigma \rrbracket.$

\begin{lemma}\label{Sigma=pirTr}
With $\Sigma,\ A_r,\ K_r,\ \pi_r$ defined as in Section \ref{3.1}, let \( T_r\in \mathcal Z_n(K_r,A_r;\mathbb Z_2) \) satisfy
\[
[T_r]=[\llbracket\Sigma\rrbracket]
\quad\text{in}\quad
H_n(K_r,A_r;\mathbb Z_2).
\]
Then we have \( (\pi_r)_{\#}T_r=\llbracket\Sigma\rrbracket \) as \(\mathbb Z_2\)-currents.
\end{lemma}

\begin{proof}
Since \( [T_r]=[\llbracket\Sigma\rrbracket]\) in \( H_n(K_r,A_r;\mathbb Z_2), \) there exist \( Q_r\in \mathcal I_{n+1}(K_r;\mathbb Z_2)\) and \(R_r\in \mathcal I_n(A_r;\mathbb Z_2) \) such that \( T_r-\llbracket\Sigma\rrbracket = \partial Q_r+R_r. \) Here and below the sign is irrelevant over \(\mathbb Z_2\).

Applying \((\pi_r)_{\#}\) to both sides gives
\[
(\pi_r)_{\#}T_r-(\pi_r)_{\#}\llbracket\Sigma\rrbracket=
\partial(\pi_r)_{\#}Q_r+(\pi_r)_{\#}R_r.
\]
Since \( \pi_r|_{\Sigma}=\operatorname{id}_{\Sigma}, \) we have \( (\pi_r)_{\#}\llbracket\Sigma\rrbracket = \llbracket\Sigma\rrbracket. \) Moreover, since \( \operatorname{spt}R_r\subset A_r \) and \( \pi_r(A_r)\subset \partial\Sigma, \) we have \( \operatorname{spt}(\pi_r)_{\#}R_r\subset\partial\Sigma. \) Thus
\[
(\pi_r)_{\#}T_r-\llbracket\Sigma\rrbracket
=
\partial(\pi_r)_{\#}Q_r+(\pi_r)_{\#}R_r,
\]
with the last term supported in \(\partial\Sigma\). Hence \( [(\pi_r)_{\#}T_r] = [\llbracket\Sigma\rrbracket]\) in \( H_n(\Sigma,\partial\Sigma;\mathbb Z_2). \)

Now \((\pi_r)_{\#}T_r\) is an \(n\)-dimensional relative cycle in \((\Sigma,\partial\Sigma)\). Since it is supported in the \(n\)-manifold \(\Sigma\), the constancy theorem implies that
\[
(\pi_r)_{\#}T_r=
\sum_{\alpha} c_{\alpha}\llbracket\Sigma_{\alpha}\rrbracket,
\qquad
c_{\alpha}\in\mathbb Z_2,
\]
where \(\{\Sigma_{\alpha}\}\) are possibly connected components of \(\Sigma\). The equality \( [(\pi_r)_{\#}T_r] = [\llbracket\Sigma\rrbracket]\) forces \( c_{\alpha}=1 \) for every component \(\Sigma_{\alpha}\). Therefore \( (\pi_r)_{\#}T_r = \sum_{\alpha}\llbracket\Sigma_{\alpha}\rrbracket = \llbracket\Sigma\rrbracket \) as desired.
\end{proof}

Lemma \ref{Sigma=pirTr} shows that projection does not lose the homology class, although \(T_r\) may be singular or non-graphical, its projection still has exactly the same mod \(2\) multiplicity as \(\Sigma\).
With Lemmas \ref{Jestimate} and \ref{Sigma=pirTr}, it suffices to show the mass convergence.

\begin{lemma}\label{massconvergence}
Let \(T_r\) be a relative mass minimizer of class \(\llbracket\Sigma\rrbracket\) in \(K_r\), Then
\[
\mathbf M(T_r)\to \mathbf M(\llbracket\Sigma\rrbracket)=
\mathcal H^n(\Sigma)
\]
as \(r\to0\). More precisely, there exists \(C>0\) independent of \(r\) such that
\[
\frac{1}{1+Cr}\mathbf M(\llbracket\Sigma\rrbracket)\le
\mathbf M(T_r)\le
\mathbf M(\llbracket\Sigma\rrbracket).
\]
\end{lemma}

\begin{proof}
Since \(T_r\) minimizes mass in its relative homology class and since \(\llbracket\Sigma\rrbracket \) itself is an admissible competitor, we immediately have
\begin{equation}\label{eq: TrleSigma}
  \mathbf M(T_r)\le \mathbf M(\llbracket\Sigma\rrbracket)
\end{equation}

We now prove the reverse inequality up to a ratio \(1+Cr\). Since \( (\pi_r)_{\#}T_r=\llbracket\Sigma\rrbracket, \) we have \( \mathbf M(\llbracket\Sigma\rrbracket) = \mathbf M((\pi_r)_{\#}T_r). \) By the mass estimate for push-forwards of rectifiable currents,
\[
\mathbf M((\pi_r)_{\#}T_r)\le
\int J_n\pi_r\bigl(y,\operatorname{Tan}(T_r,y)\bigr)\,d\|T_r\|(y),
\]
where \(\operatorname{Tan}(T_r,y)\) denotes the approximate tangent \(n\)-plane of \(T_r\) at \(y\), defined for \(\|T_r\|\)-almost every \(y\). By the Jacobian estimate from Lemma \ref{Jestimate}, \( J_n\pi_r\bigl(y,\operatorname{Tan}(T_r,y)\bigr)\le 1+Cr. \) Therefore
\begin{equation}\label{eq: TrgesSigma}
  \mathbf M(\llbracket\Sigma\rrbracket)\le
(1+Cr)\int d\|T_r\|=
(1+Cr)\mathbf M(T_r).
\end{equation}

Combining \eqref{eq: TrleSigma}, \eqref{eq: TrgesSigma} gives
\[
\frac{1}{1+Cr}\mathbf M(\llbracket\Sigma\rrbracket)\le
\mathbf M(T_r)\le
\mathbf M(\llbracket\Sigma\rrbracket).
\]
Letting \(r\to0\), we obtain \( \mathbf M(T_r)\to \mathbf M(\llbracket\Sigma\rrbracket). \)
\end{proof}

Hence by above lemma, any relative minimizer has almost the same mass as \(\Sigma\) when the tubular neighborhood is sufficiently thin.

\subsubsection*{Convergence as \texorpdfstring{$\mathbb{Z}_2$}{Z2}-Currents}

We first record the convergence in the relative flat topology. The proof uses the deformation retraction from \(K_r\) to \(\Sigma\) along the flow lines of \(Z\). Since this homotopy preserves \(A_r\), the homotopy formula gives a relative filling of \(T_r-\llbracket\Sigma\rrbracket\). The filling has small mass because the homotopy moves points by distance \(O(r)\).

\begin{lemma}\label{flatconvergence}
With above definitions, we have \(T_r\to \llbracket\Sigma\rrbracket \) in the relative flat topology of the pair \((K_r,A_r)\). Quantitatively, there exists \(C>0\), independent of \(r\), such that
\[
\mathcal F_{\mathrm{rel}}(T_r-\llbracket\Sigma\rrbracket)\le
Cr\,\mathbf M(T_r)\to0 \qquad\text{as }r\to0.
\]
\end{lemma}

\begin{proof}
Define a deformation retraction by\begin{equation}\label{eq: homotopy retraction-flatconv} H:[0,1]\times K_r\to K_r,\qquad (s,F(x,t))\mapsto F(x,(1-s)t),
\end{equation}
then \( H(0,y)=y\), \( H(1,y)=\pi_r(y). \) Moreover, \(H(s,A_r)\subset A_r\) for every \(s\in[0,1]\).

Let \( S_r=H_{\#}(\llbracket[0,1]\rrbracket\times T_r). \) By the homotopy formula for currents,
\begin{align*}
  \partial S_r&=(H_1)_{\#}T_r-(H_0)_{\#}T_r-H_{\#}(\llbracket[0,1]\rrbracket\times \partial T_r)\\
  &=(\pi_r)_{\#}T_r-T_r-H_{\#}(\llbracket[0,1]\rrbracket\times \partial T_r).
\end{align*}
By Lemma \ref{Sigma=pirTr}, \( (\pi_r)_{\#}T_r=\llbracket\Sigma\rrbracket, \) we obtain
\[
T_r-\llbracket\Sigma\rrbracket=
-\partial S_r-
H_{\#}(\llbracket[0,1]\rrbracket\times \partial T_r).
\]
Over \(\mathbb Z_2\), the signs are irrelevant. Hence \( T_r-\llbracket\Sigma\rrbracket = \partial S_r+R_r, \) where \( R_r=H_{\#}(\llbracket[0,1]\rrbracket\times \partial T_r). \)

Since \(T_r\) is a relative cycle in \((K_r,A_r)\), we have \( \operatorname{spt}\partial T_r\subset A_r. \) And since \(H(s,A_r)\subset A_r\), it follows that \( \operatorname{spt}R_r\subset A_r. \) Therefore \( T_r-\llbracket\Sigma\rrbracket = \partial S_r+R_r,\) with \( \operatorname{spt}R_r\subset A_r. \) This shows that \(T_r-\llbracket\Sigma\rrbracket\) is a relative boundary up to the filling \(S_r\).

It remains to estimate \(\mathbf M(S_r)\). We claim that \( \mathbf M(S_r)\le Cr\,\mathbf M(T_r). \) Indeed, for \(|t|\le r\) and \( H(s,F(x,t))=F(x,(1-s)t), \) differentiating in \(s\), we get
\[
\partial_s H(s,F(x,t))=
-t\,\partial_tF(x,(1-s)t).
\]
Since \(F\) is a smooth tubular parametrization and \(r\) is sufficiently small, there is a constant \(C_1>0\), independent of \(r\), such that \( |\partial_tF|\le C_1. \) Therefore \( |\partial_s H(s,F(x,t))|\le C_1|t|\le C_1r. \)

Moreover, the derivatives of \(H\) in the spatial directions are uniformly bounded in $K_r$, that is: \( |D_yH(s,F(x,t))|\le C_2. \) Hence for the \((n+1)\)-Jacobian of \(H\) on the product current \( \llbracket[0,1]\rrbracket\times T_r \),
\begin{equation}\label{eq: Jn+1H}
J_{n+1}H\le Cr.
\end{equation} 
Consequently, by the mass estimate for push-forwards,
\[
\mathbf M(S_r)=
\mathbf M\bigl(H_{\#}(\llbracket[0,1]\rrbracket\times T_r)\bigr)\le
\int_{\llbracket[0,1]\rrbracket\times T_r} J_{n+1}H\le
Cr\,\mathbf M(T_r).
\]

By the definition of the relative flat norm, we have
\[
\mathcal F_{\mathrm{rel}}(T_r-\llbracket\Sigma\rrbracket)
\le
\mathbf M(S_r)\le
Cr\,\mathbf M(T_r)\le Cr\mathbf M(\llbracket\Sigma\rrbracket)\to0 \qquad\text{as }r\to0
\]as desired.
\end{proof}

\subsubsection*{Varifold Convergence}

We next upgrade the mass and relative flat convergence to varifold convergence. The support of \(T_r\) lies in \(K_r\), and hence any varifold limit is supported on \(\Sigma\). The mass convergence shows that the limiting mass is exactly \(\mathcal H^n(\Sigma)\). It remains only to show that no tangent-plane oscillation is lost in the limit.

The first Jacobian estimate shown in Lemma \ref{Jestimate} gives mass convergence, the second estimate detects the tilt of \(T_r\) away from \(T\Sigma\): if a tangent plane is not close to \(T\Sigma\), the orthogonal projection strictly decreases its \(n\)-dimensional volume. This gives the tilt-excess convergence.

\begin{lemma}\label{tilt}
With all definitions same as in above sections, we have: 
\[
\int
\operatorname{dist}_{\mathcal G}
\left(
\operatorname{Tan}(T_r,y),
T_{\pi_r(y)}\Sigma
\right)^2
\,d\|T_r\|(y)\le Cr
\to0\qquad\text{as }r\to0.
\]
Here the constant \(C>0\) is independent of \(r\) and $\operatorname{dist}_{\mathcal G}$ is the the Grassmannian distance introduced in Lemma \ref{Jestimate}.
\end{lemma}

\begin{proof}
Apply the second estimate from Lemma \ref{Jestimate} to \(S=\operatorname{Tan}(T_r,y)\), which exists for \(\|T_r\|\)-almost every \(y\). By the push-forward mass estimate for rectifiable currents, using \( (\pi_r)_{\#}T_r=\llbracket\Sigma\rrbracket, \) we have:
\[
\mathbf M(\llbracket\Sigma\rrbracket)
=\mathbf M((\pi_r)_{\#}T_r)\le
\int J_n\pi_r
\left(
y,\operatorname{Tan}(T_r,y)
\right)
\,d\|T_r\|(y).
\]
Furthermore, by the Jacobian estimate and the choice of minimizer $T_r,$
\begin{align*}
  \mathbf M(\llbracket\Sigma\rrbracket)&\le(1+Cr)\mathbf M(T_r)-c\int\operatorname{dist}_{\mathcal G}\left(\operatorname{Tan}(T_r,y),T_{\pi_r(y)}\Sigma \right)^2 d\|T_r\|(y)\\
  &\le(1+Cr)\mathbf M(\llbracket\Sigma\rrbracket)-c\int\operatorname{dist}_{\mathcal G}\left(\operatorname{Tan}(T_r,y),T_{\pi_r(y)}\Sigma\right)^2d\|T_r\|(y).
\end{align*}
Rearranging, we obtain
\begin{equation}\label{eq: tiltexcess}
\int
\operatorname{dist}_{\mathcal G}
\left(
\operatorname{Tan}(T_r,y),
T_{\pi_r(y)}\Sigma
\right)^2
d\|T_r\|(y)\le
\frac{C}{c}\mathbf M(\llbracket\Sigma\rrbracket)r=:C'r\to0\qquad\text{as }r\to0.
\end{equation}
This proves the tilt-excess convergence.
\end{proof}

With lemma \ref{tilt} above, it suffices to prove the varifold convergence of $V_r.$

\begin{corollary}\label{vrfdconvergence}
Under the assumptions above, \( V_r:=|T_r|\to |\Sigma| \) as varifolds.
\end{corollary}

\begin{proof}
Let \(\varphi\in C_c(G_n(N))\). We need to show \( \int \varphi(y,S)\,dV_r \to \int \varphi(x,T_x\Sigma)\,d|\Sigma|. \)

By support convergence, for \(\|T_r\|\)-almost every \(y\in\operatorname{spt}T_r\), \( d_N(y,\pi_r(y))\le Cr. \) By tilt-excess convergence, \( \int \operatorname{dist}_{\mathcal G} \left( \operatorname{Tan}(T_r,y), T_{\pi_r(y)}\Sigma \right)^2 \,d\|T_r\|(y) \to0. \) Since \(\varphi\) is uniformly continuous on the compact region, it follows that
\[
\int
\left|
\varphi\left(y,\operatorname{Tan}(T_r,y)\right)-
\varphi\left(\pi_r(y),T_{\pi_r(y)}\Sigma\right)
\right|
\,d\|T_r\|(y)
\to 0.
\]
Thus it suffices to prove
\[
\int
\varphi\left(\pi_r(y),T_{\pi_r(y)}\Sigma\right)
\,d\|T_r\|(y)
\to
\int_{\Sigma}\varphi(x,T_x\Sigma)\,d\mathcal H^n(x).
\]
This follows from \( (\pi_r)_{\#}T_r=\llbracket\Sigma\rrbracket \) together with the mass convergence \( \mathbf M(T_r)\to \mathbf M(\llbracket\Sigma\rrbracket) \) and the Jacobian estimate \( J_n\pi_r\le 1+Cr. \) Indeed, the projection loses no mass in the limit, so the push-forward of the measures \( \|T_r\| \) under \(\pi_r\) converges weakly to \( \mathcal H^n\llcorner\Sigma. \) Therefore
\[
\int \varphi(y,S)\,dV_r(y,S)
\to
\int_{\Sigma}\varphi(x,T_x\Sigma)\,d\mathcal H^n(x).
\]
Hence \( V_r\to|\Sigma| \) as varifolds. Moreover, since the limiting varifold is precisely \(|\Sigma|\), rather than \(m|\Sigma|\) for some \(m\ge2\), the convergence is of multiplicity one.
\end{proof}

We have therefore shown that the relative minimizers converge to \(\Sigma\) with multiplicity one as varifolds. This is still not enough to apply regularity theory: we also need a uniform first variation bound. This is the purpose of the next subsection.

\subsection{First Variation Bound}\label{3.3}
The next goal is to prove a first variation bound for \(V_r=|T_r|\) which is uniform in \(r\). This is not immediate from the minimizing property, because \(T_r\) minimizes only among competitors supported in \(K_r\). A general free-boundary variation may move points outside \(K_r\). The idea is to compose such a variation with the Lipschitz retraction \(P_r\) constructed in Section~\ref{3.1}. This produces admissible competitors in the same relative homology class.
\begin{lemma}\label{competitor}
Let \(K_r\), \( A_r,\) defined as above, and let \( T_r\in \mathcal Z_n(K_r,A_r;\mathbb Z_2) \) be a relative cycle. Let \(X\in \mathfrak{X}^\partial(N)\), namely
\[
X(p)\in T_p\partial N\qquad p\in\partial N,
\]
and let \(\psi_t\) be the flow generated by \(X\). Define
\[
P_r:F(\Sigma\times(-r_0,r_0))\to K_r,\qquad F(x,s)\mapsto F(x,q_r(s)),
\]
as in Section \ref{3.1}. For \(|t|\) sufficiently small, set \( f_t:=P_r\circ\psi_t. \) Then \((f_t)_\#T_r\in \mathcal Z_n(K_r,A_r;\mathbb Z_2). \) Moreover,
\[
[(f_t)_\#T_r]=[T_r]
\qquad
\text{in }H_n(K_r,A_r;\mathbb Z_2).
\]
In particular, if \( [T_r]=[\llbracket\Sigma\rrbracket], \) then \( [(f_t)_\#T_r]=[\llbracket\Sigma\rrbracket]. \)
\end{lemma}

\begin{proof}
Since \(X\in \mathfrak X^\partial(N)\), its flow preserves the ambient boundary, i.e. \( \psi_t(\partial N)\subset \partial N. \) Also, by construction of the defining function \(\tau\) in section \ref{3.1}, we have \( \nabla\tau\in T\partial N\) on \(\partial N, \) and the vector field \( Z=\frac{\nabla\tau}{|\nabla\tau|^2}\) tangent to \(\partial N\) along \(\partial N\). Therefore the flow \(F\) of \(Z\) preserves \(\partial N\), and so does the retraction \(P_r\), i.e. \( P_r(\partial N)\subset \partial N. \)

Since \(P_r\) maps its domain into \(K_r\), we have \( f_t(K_r)=(P_r\circ\psi_t)(K_r)\subset K_r. \) Furthermore, if \(p\in A_r=K_r\cap\partial N\), then \( \psi_t(p)\in\partial N, \) and hence \( f_t(p)=P_r(\psi_t(p))\in \partial N, \) i.e. \( f_t(A_r)\subset A_r. \) Thus \( f_t:(K_r,A_r)\to (K_r,A_r) \) is a map of pairs.

Therefore, with above discussion, immediately we have \( (f_t)_\#T_r\in\mathcal Z_n(K_r,A_r;\mathbb Z_2). \) since the push-forward commutes with boundary.

It remains to show that \((f_t)_\#T_r\) represents the same relative homology class as \(T_r\). Define
\begin{equation}\label{eq: homotopyPr}
H:[0,1]\times K_r\to K_r \qquad\text{by }H(\lambda,p)=P_r(\psi_{\lambda t}(p)).
\end{equation}
Since \(P_r|_{K_r}=\operatorname{id}_{K_r}\), we have
\[
H(0,p)=P_r(p)=p,\quad H(1,p)=P_r(\psi_t(p))=f_t(p).
\]
Moreover, because \(\psi_{\lambda t}(\partial N)\subset\partial N\) and \( P_r(\partial N)\subset\partial N, \) we have \( H(\lambda,A_r)\subset A_r \) Thus \(H\) is a homotopy of pairs between \( \operatorname{id}_{(K_r,A_r)}\) and \(f_t.\)

By the homotopy formula, there exist
\[
H_\#\bigl([0,1]\times T_r\bigr)=:Q_t\in \mathcal I_{n+1}(K_r;\mathbb Z_2), \qquad
H_\#\bigl([0,1]\times \partial T_r\bigr)=:R_t\in \mathcal I_n(A_r;\mathbb Z_2),
\]
such that \( (f_t)_\#T_r-T_r=\partial Q_t+R_t. \) Therefore \( [(f_t)_\#T_r]=[T_r]\) in \(H_n(K_r,A_r;\mathbb Z_2). \) In particular, if \( [T_r]=[\llbracket\Sigma\rrbracket], \) then \( [(f_t)_\#T_r]=[\llbracket\Sigma\rrbracket]. \)
\end{proof}
Lemma~\ref{competitor} shows that the retracted variation produces legitimate competitors. To extract a first variation estimate from the minimizing property, we also need to know how the \(n\)-Jacobian of this retracted variation changes to first order in \(t\). This is the content of the next lemma.

The additional term \(C|X|\) below measures the cost of retracting the varied surface back into \(K_r\). It is uniform in \(r\), which is the crucial point for applying regularity later.
\begin{lemma}\label{Jf}
Let \(M\subset K_r\) be a countably \(n\)-rectifiable set, let \(X\in \mathfrak X^\partial(N)\), and let \(\psi_t\) be the flow generated by \(X\). Set \( f_t=P_r\circ\psi_t. \) Then there exists a constant \(C>0\), independent of \(r\), such that for \(\mathcal H^n\)-almost every \(y\in M\), and setting $S:=\operatorname{Tan}(M,y)$, one has
\[
J_n(Df_t(y)|_S)\le
1+t\operatorname{div}_S X(y)+C|t||X(y)|+o(t),
\]
where \(o(t)/t\to0\) as \(t\to0\), locally uniformly in \((y,S)\).
\end{lemma}

\begin{proof}
Recall that \( F:\Sigma\times(-r_0,r_0)\to \mathcal U \) is defined by \( F(x,s)=\Phi_s(x), \) where \(\Phi_s\) is the flow of \( Z={\nabla\tau}/{|\nabla\tau|^2}. \) The ambient Riemannian metric on \(N\) is denoted by \(g\). We pull it back by \(F\) and write \( \bar g:=F^*g \) for the induced metric on \(\Sigma\times(-r_0,r_0)\).

For each fixed \(s\), let
\[
F_s:\Sigma\to \Sigma_s \qquad\text{by }F_s(x)=F(x,s).
\]
Then the metric induced on \(\Sigma\) from the leaf \(\Sigma_s\) is \( g_s:=F_s^*g. \)

For each \(s\in(-r_0,r_0)\), let
\[
i_s:\Sigma\to \Sigma\times(-r_0,r_0),
\qquad
i_s(x)=(x,s)
\]
be the inclusion of the \(s\)-slice. Then \( F_s=F\circ i_s\), Therefore,
\[
g_s:=F_s^*g=
(F\circ i_s)^*g=
i_s^*(F^*g)=
i_s^*\bar g.
\]
In other words, \( g_s=F_s^*g=i_s^*\bar g. \) Thus \(g_s\) is precisely the restriction of the pulled-back metric \(\bar g\) to the slice \(\Sigma\times\{s\}\).

Since \(\partial_sF(x,s)=Z(F(x,s)), \) and since \( Z\perp T\Sigma_s, \) the pulled-back metric \(\bar g\) has no mixed terms. Thus
\[
\bar g=g_s+\lambda^2 ds^2\qquad\text{where }\lambda(x,s)=|Z(F(x,s))|_g=\frac{1}{|\nabla\tau(F(x,s))|_g}.
\]
In particular, for every \(s\),
\[
T_{F(x,s)}N=T_{F(x,s)}\Sigma_s\oplus \mathbb R Z(F(x,s))
\]
is an orthogonal decomposition with respect to the ambient metric \(g\).

We first estimate the Jacobian of the leaf-to-leaf map. For \(s,s'\in(-r_0,r_0)\), define
\[
\mathcal L_{s\to s'}:\Sigma_s\to\Sigma_{s'}\qquad\text{by }\mathcal L_{s\to s'}(F(x,s))=F(x,s').
\]
Equivalently, \( \mathcal L_{s\to s'}=F_{s'}\circ F_s^{-1}. \)

Since \(\bar g=F^*g\) is smooth on the compact set \( \Sigma\times[-r_0,r_0], \) the family \(g_s\) depends smoothly on \(s\). Hence there exists \(C_1>0\) such that for all \(\xi\in T_x\Sigma \) and all \(|s|\le r_0\),
\[
\left|
\frac{d}{ds}g_s(\xi,\xi)
\right|\le
C_1 g_s(\xi,\xi).
\]
Indeed, this follows because \(\partial_s g_s\) is a smooth family of symmetric bilinear forms on a compact set.

By Gronwall's inequality,
\[
e^{-C_1|s-s'|}g_s(\xi,\xi)
\le
g_{s'}(\xi,\xi)
\le
e^{C_1|s-s'|}g_s(\xi,\xi).
\]
Therefore, as quadratic forms on \(T_x\Sigma\), \( g_{s'}\le e^{C_1|s-s'|}g_s. \) Taking \(n\)-dimensional volume elements, using \(|s-s'|<r_0\) and increasing the constant, we obtain
\begin{equation}\label{eq: LEAFESTIMATE}
  \begin{aligned}
J_n\bigl(d\mathcal L_{s\to s'}|_{T\Sigma_s}\bigr)
&=
\frac{d\operatorname{vol}_{g_{s'}}}{d\operatorname{vol}_{g_s}}  
\le
\left(e^{C_1|s-s'|}\right)^{n/2}  
=
e^{\frac n2 C_1|s-s'|}  
\le
1+C_2|s-s'|.
\end{aligned}
\end{equation}

Next we estimate the Jacobian of \(P_r\). Write a point \(z\in \mathcal U\) as \( z=F(x,s). \) Recall that
\[
P_r(F(x,s))=F(x,q_r(s)) \qquad\text{where }q_r(s)=
\begin{cases}
-r,&s<-r,\\
s,&-r\le s\le r,\\
r,&s>r.
\end{cases}
\]
At points where \(P_r\) is differentiable, take an arbitrary vector \( v\in T_zN, \) using the orthogonal splitting \( T_zN=T_z\Sigma_s\oplus \mathbb R Z(z), \) write
\[
v=v^T+aZ(z),
\qquad
v^T\in T_z\Sigma_s.
\]
Since the splitting is orthogonal, one has \( |v^T|_g\le |v|_g. \)

If \(|s|\le r\), then \(q_r(s)=s\), so \(P_r\) is the identity near \(z\), and hence \( J_n(DP_r(z)|_S)=1. \)

If \(s>r\) or \(s<-r\), then \(q_r(s)\) is constant in \(s\), and \(DP_r\) kills the \(Z\)-component. More precisely,
\[
DP_r(z)(v)=d\mathcal L_{s\to q_r(s)}(v^T).
\]
Let \( \operatorname{pr}_s:T_zN\to T_z\Sigma_s \) be the orthogonal projection. Then \( v^T=\operatorname{pr}_s(v). \) Thus, for every \(n\)-plane \(S\subset T_zN\),
\[
J_n(DP_r(z)|_S)
\le
J_n(d\mathcal L_{s\to q_r(s)}|_{T\Sigma_s})\,
J_n(\operatorname{pr}_s|_S).
\]
Since \(\operatorname{pr}_s\) is an orthogonal projection, \(J_n(\operatorname{pr}_s|_S)\le1. \) Using \eqref{eq: LEAFESTIMATE}, we obtain
\[
J_n(DP_r(z)|_S)
\le
1+C_2|s-q_r(s)|.
\]

Define the \(\tau\)-distance from \(z=F(x,s)\) to \(K_r\) by \( d_\tau(z,K_r):=|s-q_r(s)|. \) Then the previous estimate becomes \( J_n(DP_r(z)|_S) \le 1+C_2d_\tau(z,K_r). \)

Now let \(z_t=\psi_t(y),\) \(y\in K_r. \) Immediately we have \( |\tau(y)|\le r. \) Taylor expansion gives
\[
\tau(z_t)=\tau(\psi_t(y))
=
\tau(y)+t\,d\tau_y(X(y))+o(t).
\]
Since \(d\tau\) is bounded on the fixed compact set \(\{|\tau|\le r_0\}\), there exists \(C_3>0\) such that \( |d\tau_y(X(y))| \le C_3|X(y)|. \) Therefore
\[
|\tau(z_t)-\tau(y)|
\le
C_3|t||X(y)|+o(t).
\]
Since \(\tau(y)\in[-r,r]\), the distance from \(\tau(z_t)\) to the interval \([-r,r]\) is bounded by \( |\tau(z_t)-\tau(y)|. \) Hence
\[
d_\tau(z_t,K_r)
\le |\tau(z_t)-\tau(y)|\le
C_3t|X(y)|+o(t).
\]
Applying the \(P_r\)-Jacobian estimate at \(z_t=\psi_t(y)\), with the plane \( D\psi_t(y)(S)\subset T_{z_t}N, \) we get\begin{equation}\label{eq: JACOBI1} J_n(DP_r(z_t)|_{D\psi_t(y)(S)}) \le 1+C_4|t||X(y)|+o(t).
\end{equation}

On the other hand, the standard first-order expansion for the flow \(\psi_t\) gives
\begin{equation}\label{eq: JACOBI2}
  J_n(D\psi_t(y)|_S)
=
1+t\operatorname{div}_S X(y)+o(t).
\end{equation}

Now use the chain rule for Jacobians, substituting \eqref{eq: JACOBI1}, \eqref{eq: JACOBI2}:
\begin{align*}
  J_n(Df_t(y)|_S)
  &=J_n(D(P_r\circ\psi_t)(y)|_S)\le J_n(DP_r(z_t)|_{D\psi_t(y)(S)})\,J_n(D\psi_t(y)|_S)\\
  &\le\bigl(1+C_4|t||X(y)|+o(t)\bigr)\bigl(1+t\operatorname{div}_S X(y)+o(t)\bigr)\\
  &=1+t\operatorname{div}_S X(y)+C_4|t||X(y)|+o(t)
\end{align*}
This is the desired estimate.
\end{proof}

We now combine the homological invariance of the retracted variation with the Jacobian estimate. Since \(T_r\) is minimizing in its relative class, comparing it with \((P_r\circ\psi_t)_\#T_r\) gives the desired first variation bound.

\begin{proposition}\label{1variation}
Let $T_r$ be the mass minimizer of class \([\Sigma]\) in \((K_r,A_r)\) and let \(V_r=|T_r|.\) Then there exists a constant \(C>0\), independent of \(r\), such that for every \( X\in\mathfrak X^\partial(N), \) one has
\[
|\delta V_r(X)|
\le
C\int |X|\,d\|V_r\|.
\]
\end{proposition}

\begin{proof}
Let \( \psi_t\) be the flow of \(X\), and \( f_t=P_r\circ\psi_t. \) By lemma \ref{competitor}, \( \mathbf M(T_r) \le \mathbf M((f_t)_\#T_r). \)

By the mass estimate of push-forwards and apply Lemma \ref{Jf} for $M=T_r$, we obtain,
\begin{align*}
  \mathbf M(T_r)&=\int_{G_n(N)}1\,dV_r
  \le\mathbf M((f_t)_\#T_r)
\le \int_{G_n(N)}J_n(Df_t(y)|_S)\,dV_r(y,S)\\
&\le \int_{G_n(N)}\left[1+t\operatorname{div}_S X(y)+Ct|X(y)|+o(t)\right]dV_r(y,S)
\end{align*}
hence
\[
0
\le
t\int_{G_n(N)}\operatorname{div}_S X(y)\,dV_r(y,S)
+
Ct\int |X|\,d\|V_r\|
+
o(t)\|V_r\|(K_r).
\]

Since \( \|V_r\|(K_r)=\mathbf M(T_r)\le \mathbf M(\llbracket\Sigma\rrbracket), \) the quantity \(\|V_r\|(K_r)\) is uniformly bounded in \(r\). Thus, after dividing by \(t>0\) and letting \(t\to0^+\), we get
\[
\delta V_r(X)
\ge
-C\int |X|\,d\|V_r\|.
\]

Applying the same argument to \(-X\), which also belongs to \(\mathfrak X^\partial(N)\), and combining the two inequalities yields
\[
|\delta V_r(X)|
\le
C\int |X|\,d\|V_r\|
\]as desired.
\end{proof}

This uniform first variation bound is the analytic input needed for the Allard-type regularity theorems. Together with the multiplicity-one varifold convergence proved in the previous subsection, it allows us to promote weak convergence to graphical convergence.
\subsection{Convergence as Graphs}\label{3.4}

We now apply the preceding regularity Theorems \ref{thm:allard-interior} and \ref{G} to the minimizers \(T_r\). The hypotheses of these theorems are verified using the multiplicity-one convergence from section~\ref{3.2} and the first variation bound from section~\ref{3.3}. The main conclusion is that, once \(r\) is sufficiently small, the support of \(T_r\) is not only close to \(\Sigma\) as a varifold, but also is a single free-boundary graph over \(\Sigma\).

\begin{proposition}\label{convergence as graphs}
Given \((N^{n+1},\partial N,g)\), \(\Sigma^n\subset N\), \(\tau\), \(Z\), \(F\), \(K_r\), \(A_r\) defined as in section \ref{3.1}.
Let 
\(
T_r\in \mathcal Z_n(K_r,A_r;\mathbb Z_2)
\) 
be the homologous minimizer of class $[\llbracket\Sigma\rrbracket]$ and set 
\(
V_r=|T_r|.
\) 

Assume that, as \(r\to0\), 
\(
V_r\to |\Sigma|
\) 
as varifolds with multiplicity one, and that there exists a constant \(C_0>0\), independent of \(r\), such that for every vector field 
\(
X\in \mathfrak X^\partial(N),
\) 
one has
\[
|\delta V_r(X)|
\le
C_0\int |X|\,d\|V_r\|.
\]
Then, after decreasing \(r_0\) if necessary, for every \(0<r<r_0\) there exists a function \( u_r:\Sigma\to [-r,r] \) such that
\[
\operatorname{spt}\|V_r\|=\operatorname{spt}T_r
=
\Gamma_{u_r}
:=
\{F(x,u_r(x)):x\in\Sigma\}.
\]
Moreover \( T_r=\llbracket \Gamma_{u_r}\rrbracket\) in \(\mathcal Z_n(K_r,A_r;\mathbb Z_2), \) and \( u_r\xrightarrow{C^{1}(\Sigma)}0 \) as $r\to0$.
\end{proposition}
\begin{proof}
We divide the proof into two steps.

\medskip
\noindent
\textbf{Step 1. Verification of the Regularity Hypotheses.}

The purpose of this step is to verify the hypotheses in Theorems \ref{thm:allard-interior} and \ref{G}.

\emph{a. Interior.} For an interior point $p\in\operatorname{int}\Sigma$, let $\varepsilon_{\mathrm A}>0,\,\gamma_{\mathrm A}\in(0,1)$ be the smallness threshold and the shrinking factor corresponds to $\eta_A$ in Allard's Regularity Theorem \ref{thm:allard-interior}. Choose $R_p>0$ sufficiently small such that $B_{2R_p}(p)\Subset N^\circ$, $\Sigma\cap B_{2R_p}(p)$ is $C^1$-close to $T_p\Sigma$ and $R_p<<\varepsilon_{\mathrm A}.$

By Lemma \ref{Sigma=pirTr}, \(\pi(\operatorname{spt}T_r)=\Sigma,\) hence we may choose \(p_r\in\operatorname{spt}\|V_r\|\cap \pi^{-1}(p)\) with \(d(p_r,p)\le Cr.\) Thus we have: \[ d(p_r,p)<\frac{1}{10}\gamma_{\mathrm A}R_p\qquad\text{and}\qquad B_{R_p}(p_r)\subset B_{2R_p}(p)\qquad\text{for \(r\) sufficiently small.}
  \]

The hypotheses of Allard's Regularity Theorem then follows locally in $B_{R_p}(p_r)$. Indeed, denote the affine plane through $p_r$ parallel to $T_p\Sigma$ as $P_r$, since \(\Sigma\) is smooth and \(p_r\to p\), \(\Sigma\) is arbitrarily close to this affine plane in \(B_{R_p}(p_r)\), after first choosing \(R_p\) small. Since \(V_r\to|\Sigma|\) as varifolds with multiplicity one convergence, the mass ratio of \(V_r\) in \(B_{R_p}(p_r)\) is as small as required for all sufficiently small \(r\), while the density and support condition naturally holds.

Furthermore, we have $|\delta V_r(X)|\le C_0\int |X|\,d\|V_r\|$ for all $X\in \mathfrak X^\partial(N)$ by assumption. Choose the normal coordinate centered at $p_r$: $\Psi_r:B_{2R_p}(p_r)\rightarrow B_{2R_p}(0)\subset\mathbb R^{n+1}$ and the rescale\[ \eta_{R_p}(z)=\frac{z}{R_p}, \qquad \widetilde V_r = (\eta_{R_p}\circ\Psi_r)_\#V_r.\] Under this scaling, the first variation bound becomes
\[
|\delta \widetilde V_r(Y)|
\le
C_0R_p\int |Y|\,d\|\widetilde V_r\|<\varepsilon_{\mathrm A}\int |Y|\,d\|\widetilde V_r\|
\]
for vector fields \(Y\) on the rescaled ball. Therefore the mean curvature is bounded by \( \varepsilon_{\mathrm A} \), and hence Theorem \ref{thm:allard-interior} then applies to $V_r$ in $B_{R_p}(p_r)$, therefore \( \operatorname{spt}\|V_r\|\cap B_{\gamma_{\mathrm A}R_p}(p_r) \) is a \(C^{1,\alpha_{\mathrm A}}\) graph. Consequently $\operatorname{spt}\|V_r\|\cap B_{\frac12\gamma_{\mathrm A}R_p}(p)$ is a \(C^{1,\alpha_{\mathrm A}}\) graph which is $C^1$- close to $\Sigma $ by the previous set and tilt-excess convergence in Section \ref{3.2}. Indeed, if $\operatorname{dist}_\mathcal G\bigl(T_{y}T_r,T_{\pi(y)}\Sigma\bigr)\ge \varepsilon_0$ for some $y$, then by local \(C^{1,\alpha_{\mathrm A}}\) controllness, for a sufficiently small ball the grassmannian distance is greater than $\varepsilon_0/2,$ which gives a positive local tilt-excess, contradicts to \eqref{eq: tiltexcess}.

\emph{b. Boundary.} For a boundary point \(p\in\partial\Sigma\), as in step 1.a, choose $\varepsilon_{\mathrm{GJ}},\,\eta_{\mathrm{GJ}}$ corresponds to \(\eta_{\mathrm{GJ}}\) in Gr\"uter--Jost's Regularity Theorem \ref{G} and similarly choose sufficiently small $R_p$ with extra boundary condition \(R_p\sup_{\partial N\cap B_{2R_p}(p)} |A_{\partial N}|\le\varepsilon_{\mathrm{GJ}}^2.\) Using the same argument again we obtain \(\operatorname{spt}\|V_r\|\cap B_{\frac12\gamma_{\mathrm{GJ}}R_p}(p)\) is a \(C^{1,\alpha_{\mathrm{GJ}}}\) free-boundary hypersurface which is \(C^1\)-close to \(\Sigma\). Note that at boundary we may choose Fermi coordinate instead of normal coordinate.

\medskip
\noindent
\textbf{Step 2. From local regularity to global graph.}

The regularity theorems give only local results. We must still show that these local graphs patch together into one global graph over \(\Sigma\), rather than several sheets.  We do not glue the local planar graph functions given by regularity theorems, instead, we use the tubular projection \(\pi:\mathcal U\to\Sigma. \)

Set \( \Gamma_r:=\operatorname{spt}\|V_r\|=\operatorname{spt}T_r. \) By step 1, for every \(p\in\Sigma\), there exists a neighborhood \(W_p\) of \(p\) such that, for \(r\) sufficiently small, \( \Gamma_r\cap W_p \) is a \(C^{1,\alpha}\) hypersurface, and is \(C^1\)-close to \(\Sigma\). Since \(\Sigma\) is compact, choose finitely many points \( p_1,\ldots,p_J\in\Sigma \) such that \( \Sigma\subset \bigcup_{i=1}^J W_{p_i}, \) after taking \(r\) smaller if necessary we also have \( \Gamma_r\subset \bigcup_{i=1}^J W_{p_i}. \) Thus the above local regularity covers the whole \(\Gamma_r\).

We now show that \( \pi|_{\Gamma_r}\) is a local diffeomorphism. Indeed, in the tubular coordinates \(F(x,s)\), the kernel of \(D\pi\) is precisely the line generated by \(Z\), i.e.
\[
\ker D\pi_y=\operatorname{span}\{Z(y)\}.
\]
Since \(\operatorname{dist}_\mathcal G\bigl(T_y\Gamma_r,T_{\pi(y)}\Sigma\bigr)\xrightarrow{r\to0}0\), and \(T_{\pi(y)}\Sigma\) is transverse to the \(Z\)-direction, \(T_y\Gamma_r\) is also transverse to the \(Z\)-direction. Hence
\[
D(\pi|_{\Gamma_r})_y:
T_y\Gamma_r\to T_{\pi(y)}\Sigma
\]
is an isomorphism for every \(y\in\Gamma_r\). Therefore \( \pi|_{\Gamma_r} \) is a local diffeomorphism.

Next we prove that \(\pi|_{\Gamma_r}\) is onto. By Lemma \ref{Sigma=pirTr},
\[
\Sigma=\operatorname{spt}\bigl(\llbracket\Sigma\rrbracket\bigr)=\operatorname{spt}\bigl((\pi)_\#T_r\bigr)
\subset
\pi(\operatorname{spt}T_r)
=
\pi(\Gamma_r),
\]
we obtain \( \Sigma\subset \pi(\Gamma_r). \) The reverse inclusion follows from the definition of \(\pi\). Thus \( \pi(\Gamma_r)=\Sigma. \)

Consequently, \( \pi|_{\Gamma_r}:\Gamma_r\to\Sigma \) is a proper surjective local diffeomorphism. Hence it is a finite-sheeted covering map. On each connected component \(\Sigma_j\) of \(\Sigma\), let \( m_j\in\mathbb N \) be the number of sheets over \(\Sigma_j\).

We claim that \( m_j=1\) for every connected component \(\Sigma_j\). Suppose by contradiction that \(m_j\ge2\) for some \(j\). Since each sheet is \(C^1\)-close to \(\Sigma_j\), the area of each sheet is \( (1+o(1))\mathcal H^n(\Sigma_j) \) for all sufficiently small $r.$ Hence
\[
\mathbf M(T_r)
=
\|V_r\|(N)
\ge
m_j(1+o(1))\mathcal H^n(\Sigma_j)
+
\sum_{\ell\ne j}(1+o(1))\mathcal H^n(\Sigma_\ell).
\]
If \(m_j\ge2\), then for all sufficiently small \(r\) this is strictly larger than \( \mathcal H^n(\Sigma). \) This contradicts the multiplicity-one varifold convergence \( V_r\to|\Sigma|, \) which implies \( \|V_r\|(N)\to\mathcal H^n(\Sigma). \) Therefore \( m_j=1 \) for every \(j\). Hence \( \pi|_{\Gamma_r}:\Gamma_r\to\Sigma \) is a one-sheeted covering map, and therefore a global diffeomorphism.

Define
\[
\sigma_r:\Sigma\to\Gamma_r \qquad\text{by }\sigma_r=(\pi|_{\Gamma_r})^{-1}
\]
\[
u_r:\Sigma\to[-r,r] \qquad\text{by }u_r(x)=\tau(\sigma_r(x)).
\]
Since \( \pi(\sigma_r(x))=x, \) the point \(\sigma_r(x)\) lies on the \(Z\)-flow line through \(x\). Hence there exists a unique \(s\in[-r,r]\) such that \( \sigma_r(x)=F(x,s). \) Since \( \tau(F(x,s))=s, \) this value of \(s\) is exactly \(u_r(x)\). Therefore \( \sigma_r(x)=F(x,u_r(x)). \) It follows that
\[
\Gamma_r
=
\operatorname{graph}({u_r})
=
\{F(x,u_r(x)):x\in\Sigma\}.
\]
Moreover \(u_r\) is locally \(C^{1,\alpha_\text{A}}\) in the interior and locally free-boundary \(C^{1,\alpha_\text{GJ}}\) at boundary.

Since \(T_r\) is a \(\mathbb Z_2\)-integral current supported on the hypersurface \(\Gamma_{r}\), one has \( T_r=\llbracket\Gamma_{r}\rrbracket \) as a \(\mathbb Z_2\)-current. Combining with step 1 gives $ u_r\xrightarrow{C^1(\Sigma)}0$, which completes the proof.

\end{proof}

\subsection{Local Minimizing Property}\label{3.5}
At this point the problem has been reduced to the classical graphical setting. The minimizer \(T_r\) is represented by a free-boundary graph \(\Gamma_{u_r}\) over \(\Sigma\), with \(u_r\to0\) in \(C^1\). We can therefore use the strict stability of \(\Sigma\), which gives strict local minimality for the area functional among small free-boundary graphs.
\begin{proposition}
Let \(\Sigma^n\subset (N^{n+1},g)\) be a smooth compact two-sided properly embedded free-boundary minimal hypersurface. Assume that \(\Sigma\) is strictly stable, namely its free-boundary second variation form
\[
Q(\phi,\phi)
=
\int_\Sigma
\left(
|\nabla^\Sigma\phi|^2
-
\bigl(|A_\Sigma|^2+\operatorname{Ric}_N(\nu,\nu)\bigr)\phi^2
\right)
d\mu_\Sigma
-
\int_{\partial\Sigma}
h^{\partial N}(\nu,\nu)\phi^2
d\sigma
\]
is positive definite.

Let \(\tau,F,K_r,A_r,T_r,V_r\) be same as in above sections. Suppose that, for \(r\to0\),
\[
T_r=\llbracket\Gamma_{u_r}\rrbracket,
\qquad
\Gamma_{u_r}=\{F(x,u_r(x)):x\in\Sigma\},
\qquad
u_r\to0
\quad\text{in }
C^1(\Sigma).
\]
Then, for all sufficiently small \(r\), \( T_r=\llbracket\Sigma\rrbracket. \) Consequently, \( \mathbf M(T)\ge \mathbf M(\llbracket\Sigma\rrbracket) \) for every \( T\in\mathcal Z_n(K_r,A_r;\mathbb Z_2) \) satisfying \( [T]=[\llbracket\Sigma\rrbracket]\) in \( H_n(K_r,A_r;\mathbb Z_2). \) Moreover, equality holds if and only if $T=\llbracket\Sigma\rrbracket.$ Thus \(\Sigma\) is locally area-minimizing in its relative homology class.
\end{proposition}

\begin{proof}
First introduce the graph area functional
\[
\mathcal A(u)=\mathcal H^n(\Gamma_u),
\qquad
\Gamma_u=\{F(x,u(x)):x\in\Sigma\}.
\]
Since \(F(\partial\Sigma,s)\subset\partial N\), each \(\Gamma_u\) is an admissible graph competitor for \(\|u\|_{C^1}\) small.

Since \(\Sigma\) is a free-boundary minimal hypersurface, we have:
\begin{equation*}
D\mathcal A(0)[u]=0\qquad D^2\mathcal A(0)[u,u]=Q(\zeta u,\zeta u),\qquad\text{where }Z|_\Sigma=\langle Z,\nu_\Sigma\rangle \nu_\Sigma=:\zeta \nu_{\Sigma}.
\end{equation*}
Strict stability then implies that there exists \(c_0>0\) such that \( D^2\mathcal A(0)[u,u] \ge c_0\|u\|_{H^1(\Sigma)}^2. \)

The area integrand for \(\Gamma_u\) is a smooth function of \(u\) and \(\nabla u\). Hence, for \(\|u\|_{C^1}\) sufficiently small, Taylor expansion gives
\[
\mathcal A(u)-\mathcal A(0)
=
\frac12D^2\mathcal A(0)[u,u]+R(u),
\]
with \( |R(u)| \le C\|u\|_{C^1}\|u\|_{H^1(\Sigma)}^2. \) Therefore, after decreasing the \(C^1\)-neighborhood of \(0\), we obtain
\[
\mathcal A(u)-\mathcal A(0)
\ge
c_1\|u\|_{H^1(\Sigma)}^2
\]
for some \(c_1>0\).

Now apply this to \(u=u_r\). Since \( u_r\xrightarrow{C^1}0, \) for \(r\) sufficiently small the above estimate applies. Hence
\[
\mathbf M(\llbracket\Sigma\rrbracket)\ge\mathbf M(T_r)
=
\mathcal H^n(\Gamma_{u_r})
=
\mathcal A(u_r)
\ge
\mathcal A(0)
+
c_1\|u_r\|_{H^1(\Sigma)}^2=\mathbf M(\llbracket\Sigma\rrbracket)+
c_1\|u_r\|_{H^1(\Sigma)}^2.
\]
This gives \( c_1\|u_r\|_{H^1(\Sigma)}^2\le0. \) Thus \( u_r\equiv0 \) for sufficiently small $r$. Hence \( \Gamma_{u_r}=\Sigma,\, T_r=\llbracket\Sigma\rrbracket. \)

Finally, since \(T_r\) was chosen as a mass minimizer among all currents \( T\in\mathcal Z_n(K_r,A_r;\mathbb Z_2) \) with \( [T]=[\llbracket\Sigma\rrbracket], \) we obtain, for every such \(T\), \( \mathbf M(T)\ge \mathbf M(\llbracket\Sigma\rrbracket) \) with equality holds if and only if $T=\llbracket\Sigma\rrbracket.$ This proves that \(\Sigma\) is unique local mass minimizer in its relative homology class. This completes the proof of Theorem~\ref{main}.
\end{proof}

\section{Proof of Theorem \ref{main2baby}}
\label{proof of second result}
Throughout this section, we further assume that $\Sigma$ is orientable and has positive index $k$. We fix an orientation of \(\Sigma\), and denote by \([\![\Sigma]\!]\) the associated integral relative current. The proof would be divided into two subsections, and we start with a restatement of the local $k$--parameter min--max characterization, using the notations in above sections:

\begin{theorem}\label{main22}
Let \((N^{n+1},g)\) be a compact Riemannian manifold with boundary. Let \(\Sigma^n\subset N\) be a compact, orientable, two-sided, properly embedded free-boundary minimal hypersurface.
Further assume that $\Sigma$ is nondegenerate and has positive Morse index $k$, i.e.
\[
\operatorname{index}(\Sigma)=k,
\qquad
\operatorname{nullity}(\Sigma)=0 .
\]
Then there exist \(r_0>r>0\), \(\rho>0\), and a free-boundary adapted tubular parametrization
\[
F:\Sigma\times(-r_0,r_0)\to N,\qquad K_r=F(\Sigma\times[-r,r]),
\qquad
A_r=K_r\cap\partial N .
\]
such that the following holds:
\begin{enumerate}
\item There exists a $k$-parameter family $\gamma_a$, where $a\in B_\rho^k$, of free--boundary graphs over $\Sigma$, with: \begin{itemize}
\item $\gamma_a$ homologous to $\Sigma$ for each $a.$
\item $\gamma_0=\Sigma.$
\item \(\mathbf M(\gamma_a)\le\mathbf M(\Sigma)-c|a|^2\)  for some constant $c>0$.

  \end{itemize}
\item For every $\mathcal F_{rel}$-continuous map \( \Phi:\overline{B_\rho^k}\to \mathcal Z_n(K_r,A_r;\mathbb Z) \) with \( [\Phi(a)]=[\llbracket\Sigma\rrbracket]\) in \(H_n(K_r,A_r;\mathbb Z)\) for each $a,$ and \( \Phi|_{\partial B_\rho^k}=\gamma|_{\partial B_\rho^k}, \) one has
\[
\sup_{a\in\overline{B_\rho^k}}\mathbf M(\Phi(a))
\ge
\mathbf M(\Sigma),
\]
with strict inequality unless $\Phi(a)=\llbracket\Sigma\rrbracket$ for some $a$.
\end{enumerate}

\end{theorem}

\begin{remark}
Under the two-sided assumption, $\Sigma$ is automatically orientable if $N$ is orientable. 
\end{remark}

\subsection{The Standard Unstable \texorpdfstring{$k$}{k}-Parameter Family}\label{sec: standart kfamily}
In this subsection, we construct a standard unstable $k$-parameter family, which completes the first half of Theorem \ref{main22}.

Let \( E_-=\operatorname{span}\{\varphi_1,\ldots,\varphi_k\} \) be the negative eigenspace of the free-boundary Jacobi operator, where \( \varphi_1,\ldots,\varphi_k \) are \(L^2\)-orthonormal basis on \(\Sigma\). Thus \( \int_\Sigma \varphi_i\varphi_j\,d\mu_\Sigma=\delta_{ij}, \) and the second variation form \(Q\) is negative definite on \( E_-. \)

Let \(\omega_\Sigma\) be the oriented volume form on \(\Sigma\). For \(i=1,\ldots,k\), define the relative \(n\)-form
\[
\eta_i:=\tau\pi^*(\zeta\varphi_i\omega_\Sigma)
\]
on \(K_r\). Since \(F(\partial\Sigma\times[-r,r])\subset \partial N\), we have \( D\pi(TA_r)\subset T\partial\Sigma. \) Hence \( \eta_i|_{TA_r}=0. \) Therefore \(\eta_i\) induces a well-defined functional on relative integral currents in \((K_r,A_r)\).

For \( T\in \mathcal Z_n(K_r,A_r;\mathbb Z), \) define
\[
s(T):=(s_1(T),\ldots,s_k(T))\in\mathbb R^k,\qquad\text{where }s_i(T):=(T-[\![\Sigma]\!])(\eta_i).
\]
Since \(\tau=0\) on \(\Sigma\), we have \( [\![\Sigma]\!](\eta_i)=0, \) and hence \( s_i(T)=T(\eta_i). \)

Now choose \(\rho>0\) sufficiently small so that, for every \( a=(a_1,\ldots,a_k)\in \overline B_\rho^k\subset \mathbb R^k, \) the function \( u_a:=\zeta^{-1}\sum_{i=1}^k a_i\varphi_i \) satisfies \( |u_a|<r. \) Denote $v_a=\zeta u_a$ and define
\[
\Gamma_a:=\{F(x,u_a(x)):x\in\Sigma\}.
\]
We orient \(\Gamma_a\) by requiring that \( \pi|_{\Gamma_a}:\Gamma_a\to\Sigma \) is orientation-preserving, and set \( \gamma(a)=\gamma_a:=[\![\Gamma_a]\!]. \) Then \( \gamma:\overline B_\rho^k\to \mathcal Z_n(K_r,A_r;\mathbb Z) \) is the standard unstable \(k\)-disk associated with \(\Sigma\).

For every \(a\in\overline B_\rho^k\), one has \( [\gamma_a]=[\llbracket\Sigma\rrbracket]\) in \(H_n(K_r,A_r;\mathbb Z). \) Moreover, \( s_i(\gamma_a)=a_i\) for \(i=1,\ldots,k. \) Indeed, denote \( F_a(x):=F(x,u_a(x)) \) as the graph function on $\Sigma,$ then
\[
F_a^*\tau=u_a,
\qquad
F_a^*\pi^*(\varphi_i\omega_\Sigma)=\varphi_i\omega_\Sigma.
\]
Hence
\[
s_i(\gamma_a)
=
\int_{\Gamma_a}\eta_i
=
\int_\Sigma F_a^*\eta_i
=
\int_\Sigma \zeta u_a\varphi_i\,d\mu_\Sigma
=
\int_\Sigma \left(\sum_{j=1}^k a_j\varphi_j\right)\varphi_i\,d\mu_\Sigma
=
a_i.
\]
Therefore \(s(\gamma_a)=a. \)

Finally, by the expansion at the free-boundary minimal hypersurface \(\Sigma\),
\[
\mathbf M(\gamma_a)
=
\mathbf M(\Sigma)
+
\frac12 Q(v_a,v_a)
+
O(|a|^3).
\]
Since \(Q\) is negative definite on \(E_-\), after decreasing \(\rho>0\) if necessary, there exists \(c>0\) such that \( \mathbf M(\gamma_a) \le \mathbf M(\Sigma)-c|a|^2 \) for all \(a\in\overline B_\rho^k\). In particular, \( \mathbf M(\gamma_a)<\mathbf M(\Sigma)\) for all \(a\in\partial B_\rho^k. \)

\subsection{Modified Mass and Modified Local Minimality Lemma}
In this subsection, we will show that every filling of the boundary condition $\gamma|_{\partial B_\rho^k}$ must reach $\mathbf M(\Sigma)$, that would complete the proof.

We write the corresponding eigenvalues of the second variation form as \( Q(\varphi_i,\varphi_j)=\lambda_i\delta_{ij},\) where \( \lambda_i<0\) for \(i=1,\ldots,k. \) Since the nullity of \(\Sigma\) is zero, the quadratic form \(Q\) is positive on the \(L^2\)-orthogonal complement of \(E_-\). Thus there exists \(c_+>0\) such that
\[
Q(h,h)\ge c_+\|h\|_{H^1(\Sigma)}^2\qquad\text{for }h\perp_{L^2}E_-.
\]

Let \(u\) be a sufficiently small graph function over \(\Sigma\) and $v=:\zeta u$. We decompose
\[
v=v_-+v_+,
\qquad\text{where }
v_-=\sum_{i=1}^k a_i\varphi_i,
\quad
v_+\perp_{L^2}E_-.
\]
By previous definition, one has \( s_i([\![\Gamma_u]\!]) = \int_\Sigma v\varphi_i\,d\mu_\Sigma = a_i. \) Hence \( |s([\![\Gamma_u]\!])|^2 = \sum_{i=1}^k a_i^2. \) Thus we could choose a constant \(\Lambda>0\) such that \( \Lambda> \max_{1\le i\le k} \left(-\frac12\lambda_i\right), \) and hence \( \Lambda+\frac12\lambda_i>0\) for all \( i=1,\ldots,k. \) By the expansion of area functional and the decomposition, we have:
\begin{align*}
  \mathbf M([\![\Gamma_u]\!])
  &=\mathbf M([\![\Sigma]\!])+\frac12 Q(v,v)+o(\|v\|_{H^1(\Sigma)}^2)\\
  &=\mathbf M([\![\Sigma]\!])+\frac12\big(\sum_{i=1}^k \lambda_i a_i^2+Q(v_+,v_+)\big)+\Lambda |s([\![\Gamma_u]\!])|^2-\Lambda |s([\![\Gamma_u]\!])|^2+o(\|v\|_{H^1(\Sigma)}^2)\\
  &\ge\mathbf M([\![\Sigma]\!])+c_\Lambda\bigl(\sum_{i=1}^k a_i^2\bigr)+c_+\|v_+\|_{H^1(\Sigma)}^2-\Lambda |s([\![\Gamma_u]\!])|^2+o(\|v\|_{H^1(\Sigma)}^2)\\
  &\ge\mathbf M([\![\Sigma]\!])+c_{E_-}\|v_-\|_{H^1(\Sigma)}^2+c_+\|v_+\|_{H^1(\Sigma)}^2-\Lambda |s([\![\Gamma_u]\!])|^2+o(\|v\|_{H^1(\Sigma)}^2)\\
  &\ge\mathbf M([\![\Sigma]\!])+c\|v\|_{H^1(\Sigma)}^2+o(\|v\|_{H^1(\Sigma)}^2)-\Lambda |s([\![\Gamma_u]\!])|^2.
\end{align*}
Note that the second inequality above comes from the norm equivalence on a finite dimensional space $E_-$, while the third inequality comes from the triangle inequality. Equivalently, define the modified mass functional by \( \mathbf M^*(T):=\mathbf M(T)+\Lambda |s(T)|^2, \) then the inequality becomes
\[
\mathbf M([\![\Gamma_u]\!])+\Lambda |s([\![\Gamma_u]\!])|^2-\mathbf M([\![\Sigma]\!])=\mathbf M^*([\![\Gamma_u]\!])-\mathbf M^*([\![\Sigma]\!])
\ge
c\|v\|_{H^1(\Sigma)}^2\ge c'\|u\|_{H^1(\Sigma)}^2
\] after absorbing the higher order term in a sufficiently small neighborhood of $\Sigma$. In particular, equality holds if and only if $u\equiv0,$ thus \([\![\Sigma]\!]\) is the strict local minimizer of \(\mathbf M^*\) among sufficiently small admissible graphs over \(\Sigma\). 

The key point of our argument is to prove that this graphical minimality property actually holds for all nearby relative integral currents in the same relative homology class. More precisely, we shall prove the following modified local minimality lemma:
\begin{lemma}\label{Modified local minimality}
For some sufficiently small $r$, and for every $T\in\mathcal Z_n(K_r,A_r;\mathbb Z)$ with $[T]=[\llbracket\Sigma\rrbracket]\in H_n(K_r,A_r;\mathbb Z)$, we have: 
\[
\mathbf M^*(T)\ge \mathbf M^*([\![\Sigma]\!])=
\mathbf M(\Sigma).
\]
Equality holds if and only if $T=\llbracket\Sigma\rrbracket$. In particular, if \( s(T)=0, \) then \( \mathbf M(T)\ge \mathbf M(\Sigma). \)
\end{lemma}
\begin{proof}
We repeat the argument in Section \ref{3} with slight modifications and retain signs and notations. Suppose the conclusion fails, then we directly choose a minimizer $\widehat T_r$ of modified mass $\mathbf M^* $ in $K_r,$ where existence guaranteed by the $\mathcal F_{rel}$--continuity of $s.$ Then the result of Section \ref{3.2} (homology invariant under projection, mass convergence, current convergence, multiplicity one varifold convergence) still holds in coefficient $\mathbb Z.$

For the uniform first variation bound in Section~\ref{3.3}, we only need to additionally show that \[ \left||s(f_{t\#} \widehat T_r)|^2-|s(\widehat T_r)|^2\right|\le Ct\int |X|d\|\widehat V_r\|+o(t) \] for $t$ sufficiently small. We prove this inequality as a claim.

\begin{proof}[Proof of claim]
Recall the homotopy retraction \eqref{eq: homotopyPr} and using the homotopy formula would give:
  \[
    f_{t\#}\widehat T_r-\widehat T_r=\partial \big((H_t)_{\#}(\llbracket[0,1]\rrbracket\times \widehat T_r)\big)+(H_t)_{\#}(\llbracket[0,1]\rrbracket\times \partial \widehat T_r):=\partial Q_t+R_t.
    \]
Since $\eta_i|_{TA_r}=0,$ we have:
\[
  s_i(f_{t\#}\widehat T_r)-s_i(\widehat T_r)=(f_{t\#}\widehat T_r-\widehat T_r)(\eta_i)=\partial Q_t(\eta_i)+R_t(\eta_i)= Q_t(d\eta_i), 
\]
where all $\eta_i$ are defined in fixed $K_{r_0}$ and $d\eta_i=d(\tau\pi^*(\zeta \phi_i\omega_{\Sigma}))=d\tau\wedge\pi^*(\zeta \phi_i\omega_{\Sigma}).$ Thus $\|d\eta_i\|_{\infty}\le C, $ where $C$ independent of $r,$ and hence $|s_i(f_{t\#}\widehat T_r)-s_i(\widehat T_r)|<C\mathbf M(Q_t)$.

On the other hand, by our construction of $P_r$ in Section \ref{3.1},\[ |\partial_\lambda H_t(\lambda,y)|=|tDP_r\bigl(X(\psi_{\lambda t}(y))\bigr)|\le C|t||X(\psi_{\lambda t}(y))|.
  \]
For spatial direction, we also have $|D_yH_t|=|DP_r(\psi_{\lambda t}(y))\circ D\psi_{\lambda t}(y)|\le C.$ Consequently we have $J_{n+1}H_t\big|_{\mathbb R\oplus\operatorname{Tan}(T,y)}\le C|t||X(\psi_{\lambda t}(y))|.$ Combine this with area formula would give \( \mathbf M(Q_t) \le C|t| \int |X|\,d\|\widehat T_r\|. \) Thus, denote $\Delta_t s=s(f_{t\#}\widehat T_r)-s(\widehat T_r)$ would have $|\Delta_t s|\le C|t|\int |X|\,d\|\widehat T_r\|.$ By our choice of $\widehat T_r,$ $|s(\widehat T_r)|\le \sqrt{\frac{\mathbf M(\Sigma)}{\Lambda}},$ and hence \begin{align*} \left||s(f_{t\#}\widehat T_r)|^2-|s(\widehat T_r)|^2\right|&\le 2|s(\widehat T_r)||\Delta_t s|+|\Delta_t s|^2\le C|t|\int |X|d\|\widehat T_r\|+Ct^2(\int |X|d\|\widehat T_r\|)^2\\ &= C|t|\int |X|d\|\widehat T_r\|+o(|t|).
      \end{align*}
This concludes the claim.\end{proof}

With the uniform variation bound, the argument of Section \ref{3.4} then applies with only notation changes and forces each minimizer $\widehat T_r$ to become a graph over $\Sigma. $ Combining previous discussion, $\Sigma$ minimizes modified mass $\mathbf M^* $ among all admissible graphs in $K_r$, and hence it is also the strict minimizer of $\mathbf M^* $ among all integral relative currents in $K_r$ as desired.
\end{proof}

We conclude that $\mathbf M(T)\ge\mathbf M(\Sigma )$ if $s(T)=0$ from above lemma. It suffices to complete the proof of Theorem~\ref{main22}. For a continuous family $\Phi: \overline B_\rho^k\to \mathcal Z_n(K_r,A_r;\mathbb Z)$ with $\Phi(a)=\gamma_a $ at $\partial B_\rho^k,$ $s\circ\Phi:\overline B_\rho^k\to\mathbb R^k$ satisfies $s(\Phi(a))=s(\gamma_a)=a$ at partial. Thus $s\circ\Phi|_{\partial B_\rho^k}=\operatorname{id}_{\partial B_\rho^k}.$ By continuity and degree theory, there exists some point $a_0\in B_\rho^k$ such that $s(\Phi(a_0))=0.$ By Lemma \ref{Modified local minimality}, $\mathbf M(\Phi(a_0))\ge \mathbf M(\Sigma)$, thus $\sup_{a\in\overline B_\rho^k}\mathbf M(\Phi(a))\ge\mathbf M(\Sigma)$ as desired. Note that equality can occur only if $\Phi(a_0)=\llbracket\Sigma\rrbracket$ for some $a_0$, this completes the proof of Theorem \ref{main22}.

\section{Flat-Neighborhood Local Minimality and Application to Min-max Theory}\label{flatversion}

\subsection{From Tubular to Flat Neighborhoods}

The tubular-neighborhood version of variational properties proved previously are the natural local statement from the viewpoint of the proof, since the projection, retraction, and the first-variation estimates are all constructed inside \(K_r\). However, from the viewpoint of Almgren--Pitts min--max theory, the tubular support condition is not the most natural one. Free-boundary sweepouts in \cite{FBMHI, FBMHII, FBMH1} are continuous in the $\mathcal F_{rel}$-topology, and a slice which is close to \(\llbracket\Sigma\rrbracket\) in this topology need not to have support contained in a certain tubular neighborhood, see Figure \ref{fig:flat-localization} for an example:

\begin{figure}[H]
\centering
\begin{tikzpicture}[x=1cm,y=1cm]

  \begin{scope}
    \path[fill=fbblue!8] (1.00,.25) rectangle (2.30,3.45);
    \draw[auxiliary] (1.00,.25)--(1.00,3.45);
    \draw[auxiliary] (2.30,.25)--(2.30,3.45);
    \draw[maincurve] (1.65,.25)
      .. controls (1.45,1.10) and (1.88,1.85) .. (1.62,2.55)
      .. controls (1.50,2.95) and (1.70,3.20) .. (1.65,3.45)
      node[above] {$\Sigma$};
    \draw[line width=1pt,draw=fbred]
      (1.77,.25)
      .. controls (1.58,.95) and (1.88,1.35) .. (1.73,1.62)
      .. controls (1.92,1.69) and (2.35,1.68) .. (2.58,1.76)
      .. controls (2.72,1.81) and (2.72,1.91) .. (2.58,1.96)
      .. controls (2.34,2.04) and (1.92,2.01) .. (1.73,2.10)
      .. controls (1.67,2.23) and (1.78,2.38) .. (1.73,2.50)
      .. controls (1.60,2.90) and (1.82,3.18) .. (1.77,3.45);
    \node[fbblue,anchor=west] at (2.38,.62) {$K_r$};
    \node[fbred,anchor=west] at (2.70,1.86) {$S$};
    \node[align=center] at (2.10,-.35)
      {$\mathcal F_{\mathrm{rel}}(S-\Sigma)$ small,\\but $\operatorname{spt}S\not\subset K_r$};
  \end{scope}
\end{tikzpicture}
\caption{Relative-flat closeness does not imply support closeness.}
\label{fig:flat-localization}
\end{figure}

The goal of this section is to remove this support restriction. For similar result for fixed boundary or boundaryless setting, see \cite{IM2018}. The argument follows the exterior-replacement strategy of Marques--Neves \cite[Section 6]{MN2021}, we shall keep the notation from Theorem~\ref{main1precise}, and denote \(\mathcal F_{\mathrm{rel}}^N\) for the relative flat distance in the ambient pair \((N,\partial N)\). We first state the theorem:

\begin{proposition}\label{thm:flat-neighborhood-local-minimality}
Let \((N^{n+1},g)\) be a compact smooth Riemannian manifold with boundary, and let \( \Sigma^n\subset N \) be a compact, orientable, two-sided, properly embedded free-boundary minimal hypersurface. Fix an orientation of $\Sigma$.

Pick smooth $n$-forms $\{\omega_i\}_{i=1}^m\in\Omega^n(N)$ satisfies $\iota^*_{\partial N}\omega_i=0$. For $T\in\mathcal Z_n(N,\partial N;\mathbb Z)$ , define \[ s_i(T):=(T-\llbracket\Sigma\rrbracket)(\omega_i), \qquad \mathfrak s(T)=(s_1(t),\cdots,s_m(T)).
  \]
and\[ \mathcal M_{\Lambda,\omega}(T)=\mathbf M(T)+\Lambda|\mathfrak s(T)|^2,\qquad \Lambda\ge 0.
  \]

Suppose that for some $r>0$, $\llbracket\Sigma\rrbracket$ is the unique relative homologous minimizer of $\mathcal M$ in $K_r$, i.e. \[ \mathcal M_{\Lambda,\omega}(T)\ge\mathcal M_{\Lambda,\omega}(\llbracket\Sigma\rrbracket)=\mathbf M(\llbracket\Sigma\rrbracket),\quad \text{for all }T\in\mathcal Z_n(K_r,A_r;\mathbb Z),\,[T]=[\llbracket\Sigma\rrbracket] \in H_n(K_r,A_r;\mathbb Z) \]with equality holds if and only if $T=\llbracket\Sigma\rrbracket$.

Then there exists \(\varepsilon>0\) such that the following holds: If
\[
S\in\mathcal Z_n(N,\partial N;\mathbb Z)
\qquad
[S]=[\llbracket\Sigma\rrbracket]\in H_n(N,\partial N;\mathbb Z)\qquad\text{and }\mathcal F_{\mathrm{rel}}^N(S-\llbracket\Sigma\rrbracket)<\varepsilon
\]
then \( \mathcal M_{\Lambda,\omega}(S)\ge \mathcal M_{\Lambda,\omega}(\llbracket\Sigma\rrbracket)=\mathbf M(\llbracket\Sigma\rrbracket). \) Equality holds if and only if \( S=\llbracket\Sigma\rrbracket \). In particular, proposition holds for coefficient $\mathbb Z_2$ when $m=0$.
\end{proposition}

\begin{proof}
We shall omit subscript $\Lambda,\omega$ below.

\medskip\noindent\textbf{Step1. }We first show that $\mathfrak s$ is continuous function under $\mathcal F_{rel}$-topology. Suppose $T-S=P+\partial Q+R$, with $\operatorname{spt} R\in \partial N$. One has $R(\omega_i)=0$, thus \begin{equation}\label{eq: LIPcontinuity}
|s_i(T)-s_i(S)|=|P(\omega_i)+Q(d\omega_i)|\le \|\omega_i\|_{\infty}\mathbf M(P)+\|d\omega_i\|_{\infty}\mathbf M(Q)\le C_i \mathcal F_{rel}^N(T-S)
\end{equation}
hence $\mathfrak s$ is $\mathcal F_{rel}$-Lipschitz.

\medskip\noindent\textbf{Step2. }Suppose by contradiction that the conclusion fails. Then there exists a sequence 
\( \{S_j\in \mathcal Z_n(N,\partial N;\mathbb Z)\}_{j=1}^\infty \) such that
\begin{equation}\label{eq: flatconvergence Sj}
[S_j]=[\llbracket\Sigma\rrbracket]
\quad\text{in }H_n(N,\partial N;\mathbb Z),\qquad\mathcal F_{\mathrm{rel}}^N(S_j-\llbracket\Sigma\rrbracket)\to0,
\end{equation}
but \( S_j\ne \llbracket\Sigma\rrbracket,\) with \( \mathbf M(S_j)\le\mathcal M(S_j)\le \mathbf M(\llbracket\Sigma\rrbracket). \) By lower semicontinuity of mass,
\[
\mathbf M(\llbracket\Sigma\rrbracket)
\le
\liminf_{j\to\infty}\mathcal M(S_j)
\le
\limsup_{j\to\infty}\mathcal M(S_j)
\le
\mathbf M(\llbracket\Sigma\rrbracket),
\]this gives \begin{equation}\label{eq:massconvergence Sj}
  \mathbf M(S_j)\to\mathbf M(\Sigma)
\end{equation}

Take $V:=F(\Sigma\times (-\rho,\rho))$, $0<\rho<r$, \eqref{eq: flatconvergence Sj} and \eqref{eq:massconvergence Sj} gives \(\ \|S_j\|\rightharpoonup^*\|\Sigma\|\) and thus $\|S_j\|(N\setminus V)\to 0$.

\medskip\noindent\textbf{Step3. }Define class\[
\mathscr{C}_j:=
\left\{
T\in\mathcal Z_n(N,\partial N;\mathbb Z):
\begin{array}{l}
{}[T]=[\llbracket\Sigma\rrbracket]\in H_n(N,\partial N;\mathbb Z),\\
\operatorname{spt}(T-S_j)\subset N\setminus V
\end{array}
\right\}.
\]Choose a $\mathcal M$-minimizer $T_j$ in class $\mathscr C_j$. Existence comes from compactness, lower semi-continuity of mass and continuity of $\mathfrak s$. Furthermore, by $T_j=S_j$ in $V$ and $\mathcal M(T_j)\le\mathcal M(S_j)\le \mathbf M(\Sigma)$:\begin{equation}\label{eq: Tj IN V}
  \|T_j\|(N\setminus V)=\mathbf M(T_j)-\|S_j\|(V)\le \mathcal M(T_j)-\|S_j\|(V)\le \mathbf M(\Sigma)-\|S_j\|(V)\to 0
\end{equation}
and hence\begin{equation*}
  \mathbf M(T_j-S_j)\le(\|T_j\|+\|S_j\|)(N\setminus V)\to 0
\end{equation*}
Thus, $\mathcal F^N_{rel}(T_j-\llbracket\Sigma\rrbracket)\to 0$ and also $\mathfrak s(T_j)\to 0$ from \eqref{eq: LIPcontinuity}.

\medskip\noindent\textbf{Step4. }Let $X\in \mathfrak X^\partial (N)$ with $\operatorname{spt}X\Subset N\setminus \bar V$, denote its flow by $\psi_t$. For $t$ small, $\psi_{t\#}(T_j)\in\mathscr{C}_j$ and hence $t=0$ is minimal point for $t\mapsto \mathcal M(\psi_{t\#}T_j)$. Thus, 
\begin{equation}
\begin{aligned}
  0&=\delta|T_j|(X)+2\Lambda\sum_{i=1}^ms_i(T_j)\left.\frac d{dt}\right|_{t=0}(\psi_t)_\#T_j(\omega_i)
  =\delta|T_j|(X)+2\Lambda\sum_{i=1}^ms_i(T_j)T_j(\mathcal L_X\omega_i)\\
  &=\delta|T_j|(X)+2\Lambda\sum_{i=1}^ms_i(T_j)T_j(d\iota_X\omega_i+\iota_Xd\omega_i)\\
  &=\delta|T_j|(X)+2\Lambda\sum_{i=1}^ms_i(T_j)(\partial T_j(\iota_X\omega_i)+T_j(\iota_Xd\omega_i))\\
  &=\delta|T_j|(X)+2\Lambda\sum_i s_i(T_j)T_j(\iota_Xd\omega_i)
\end{aligned}
\end{equation}
and \begin{equation}\label{eq: BOUNDED VARIATION Tj}
  |\delta|T_j|(X)|\leq C|\mathfrak s(T_j)|\int |X|\,d\|T_j\|\leq C\int |X|\,d\|T_j\|
\end{equation}
Thus the first variation is uniformly bounded.

\medskip\noindent\textbf{Step5. }Suppose that there exists $y_j\in\operatorname{spt}T_j\setminus K_r$. Since $\operatorname{dist}(N\setminus K_r, \bar V)>0$, could choose $\sigma>0$ independent of $j$, such that $B_{2\sigma}(y_j)\cap N\subset \setminus \bar V$. Using \eqref{eq: BOUNDED VARIATION Tj} and bounded mean-curvature free-boundary monotonicity formula \cite[Corollary 4.9]{MASI}, we obtain a uniform lower bound $\|T_j\|(B_\sigma(y)\cap N)\ge c\sigma^n>0$. This contradicts \eqref{eq: Tj IN V}, hence $\operatorname{spt}T_j\subset K_r$.

By constancy theorem and convergence $T_j\xrightarrow{\mathcal F_{rel}^N,\,\mathbf M} \llbracket\Sigma\rrbracket$, it follows that $\pi_\#T_j=\llbracket\Sigma\rrbracket$. Furthermore, using \eqref{eq: homotopy retraction-flatconv}, one has\[
[T_j]=[\llbracket\Sigma\rrbracket]
\quad\text{in }H_n(K_r,A_r;\mathbb Z)\]Apply the $\mathcal M$-minimality of $\Sigma$ gives $\mathbf M(\Sigma)\le\mathcal M(T_j)\leq\mathcal M(S_j)\leq\mathbf M(\Sigma)$ and that \(T_j=\llbracket\Sigma\rrbracket\). Since \(T_j=S_j\) in \(V\), we have \(S_j=\llbracket\Sigma\rrbracket\) in \(V\), thus \(\mathbf M(S_j)=\mathbf M(\Sigma)+\|S_j\|(N\setminus V)\). But \(\mathbf M(S_j)\leq\mathcal M(S_j)\leq\mathbf M(\Sigma)\), we imply that $S_j$ has no mass outside $V$, and \(S_j=\llbracket\Sigma\rrbracket\), contradicts with our assumption. This completes the proof. 
\end{proof}
\begin{corollary}\label{cor: two enhance}
\begin{enumerate}
  \item Under assumptions in Theorem \ref{main1precise}, then $\Sigma $ is, in fact, a unique homologous $\mathbf M$-minimizer in $H_n(N,\partial N;\mathbb Z_2)$ locally in a small $\varepsilon$-neighborhood of $\mathcal F_{rel}^N$-topology.
  \item Under assumptions in Lemma \ref{Modified local minimality}, then $\Sigma $ is, in fact, a unique homologous $\mathbf M^*$-minimizer in $H_n(N,\partial N;\mathbb Z)$ locally in a small $\varepsilon$-neighborhood of $\mathcal F_{rel}^N$-topology.
\end{enumerate}
\end{corollary}
\begin{proof}
  These are direct consequences of Proposition \ref{thm:flat-neighborhood-local-minimality} and we shall omit the proof.
\end{proof}
\begin{theorem}\label{main2flat}
  Under the hypotheses of Theorem \ref{main22} and let $\varepsilon$ be the constant of Corollary \ref{cor: two enhance}. The second result of Theorem \ref{main22} could be enhanced from tubular to $\mathcal F^N_{rel}$-neighborhood, that is: 
  
For every $\mathcal F^N_{rel}$-continuous map 
\(
\Phi:\overline{B_\rho^k}\to \mathcal Z_n(N,\partial N;\mathbb Z)
\) 
with 
\(
[\Phi(a)]=[\llbracket\Sigma\rrbracket]\) in \(H_n(N,\partial N;\mathbb Z)\), 
\(
\Phi|_{\partial B_\rho^k}=\gamma|_{\partial B_\rho^k},
\) and \(\Phi(\overline{B_\rho^k})
\subset B_{\varepsilon}^{\mathcal F^N_{{rel}}}
(\llbracket\Sigma\rrbracket)\), 
one has
\[
\sup_{a\in\overline{B_\rho^k}}\mathbf M(\Phi(a))
\ge
\mathbf M(\Sigma),
\]
with strict inequality unless $\Phi(a)=\llbracket\Sigma\rrbracket$ for some $a$.
\end{theorem}
\begin{proof}
Take cutoff function \[\chi\in C_c^\infty \bigl(F(\Sigma\times(-r_0,r_0))\bigr), \qquad\begin{cases} \chi\equiv1&\text{on a neighborhood of }K_r\\ \chi\equiv0&\text{outside the neighborhood }
\end{cases}\] Define \(\bar\eta_i=\chi\tau\pi^*(\zeta \varphi_i\omega_\Sigma)\), then \[\iota_{\partial N}^{*}\bar\eta_i=0,
\qquad \bar\eta_i=\eta_i\quad\text{on }K_r, \qquad \llbracket\Sigma\rrbracket(\bar\eta_i)=0\]Define \(\bar s_i(T):=T(\bar\eta_i)\). Lemma \ref{Modified local minimality} still applies and give that \(\mathbf M(T)+\Lambda|\bar{s}(T)|^2\geq\mathbf M(\Sigma)\). Apply Proposition \ref{thm:flat-neighborhood-local-minimality}, previous inequality holds for all current $T$ satisfies\begin{itemize}
\item $T$ homologous with \(\Sigma\) in $H_n(N,\partial N;\mathbb Z)$
\item \(\mathcal F^N_{{rel}}(T-\llbracket\Sigma\rrbracket)<\varepsilon\)
\end{itemize}
equality holds if and only if \(T=\llbracket\Sigma\rrbracket\).

Meanwhile, the standard unstable $k$-disk \(\gamma_a\subset K_r\), \(\bar s_i(\gamma_a)=s_i(\gamma_a)=a_i\). Thus, for any $\mathcal F^N_{rel}$-continuous map \(\Phi:B_\rho^k\to\mathcal Z_n(N,\partial N;\mathbb Z)\) satisfies\begin{itemize}
\item \(\Phi|_{\partial B_\rho^k}=\gamma|_{\partial B_\rho^k}\)
\item $[\Phi(a)]=[\llbracket\Sigma\rrbracket]$
\item \(\Phi(B_\rho^k)\subset B_{\varepsilon}^{\mathcal F^N_{{rel}}}(\llbracket\Sigma\rrbracket)\)
\end{itemize}
degree theory gives \(\bar{s}(\Phi(a_0))=0\) for some point $a_0$. Thus \(\mathbf M(\Phi(a_0))=\mathbf M(\Phi(a_0))+\Lambda|\bar{s}(\Phi(a_0))|^2\geq\mathbf M(\Sigma)\), and hence \(\sup_{a\in B_\rho^k}\mathbf M(\Phi(a))\geq\mathbf M(\Sigma).\)
\end{proof}

For the rest of the paper, since the restriction of certain tubular neighborhood has been removed, we shall always talk about relative flat distance under pair $(N,\partial N)$ and omit superscript $N$ of $\mathcal F^N_{rel}$.

We also record the following free-boundary local min-max theorem. This could be view as a corollary of Theorem \ref{main2flat}. For boundaryless local min-max theorem, see \cite[Theorem 6.1]{MN2021}.

\begin{theorem}
\label{thm:free-boundary-local-minmax}

Under the hypotheses of Theorem \ref{main22}, for every \(\beta>0\), after possibly decreasing \(\rho>0\), there exist constant $\varepsilon_0>0$ and a smooth family of boundary-preserving diffeomorphisms \( \{\Psi_a\}_{a\in\overline{B}_\rho^k} \subset \operatorname{Diff}(N) \) with the following properties:

\begin{enumerate}
\item $\Psi_0=\operatorname{Id}$, $\Psi_{-a}=\Psi_a^{-1}$ for every $a\in\overline{B}_\rho^k$.

\item Writing \(\gamma_a:=(\Psi_a)_\#\llbracket\Sigma\rrbracket, \) one has \(\bar s(\gamma_a)=a, \) and the function
\[
 A^\Sigma:\overline{B}_\rho^k\longrightarrow[0,\infty),
 \qquad
 A^\Sigma(a):=\mathbf M(\gamma_a)
              =\mathbf M\bigl((\Psi_a)_\#
                \llbracket\Sigma\rrbracket\bigr),
\]
is strictly concave. More precisely, there exists \(c_0>0\) such that \( D^2A^\Sigma(a)(u,u) \leq -c_0|u|^2 \) for every \(a\in\overline{B}_\rho^k\) and \(u\in\mathbb R^k\).

\item \(\|\Psi_a-\operatorname{Id}\|_{C^1}<\beta\) for every \(a\in\overline{B}_\rho^k.\)

\item For every \(S\in\mathcal Z_n(N,\partial N;\mathbb Z) \) satisfying \( [S]=[\llbracket\Sigma\rrbracket]\) in \(H_n(N,\partial N;\mathbb Z)\) and \( \mathcal F_{\mathrm{rel}} (S-\llbracket\Sigma\rrbracket)<\varepsilon_0, \) one has
\[
 \max_{a\in\overline{B}_\rho^k}
 \mathbf M\bigl((\Psi_a)_\#S\bigr)
 \geq
 \mathbf M(\Sigma).
\]
Moreover, equality can hold only if \( \llbracket\Sigma\rrbracket = (\Psi_{a_0})_\#S \) for some \(a_0\in\overline{B}_\rho^k\).
\end{enumerate}
\end{theorem}

\begin{proof}
Recall that $E^-=\operatorname{span}\{\varphi_1,\ldots,\varphi_k\}$ and $v_a=\zeta u_a=\sum_{i=1}^k a_i\varphi_i$ as in Section \ref{sec: standart kfamily}. Also recall that $Z=\frac{\nabla\tau}{|\nabla\tau|^2},\,Z(\tau)=1$, and $Z(\pi)=0$ as in Section \ref{3.1}.

Choose a smooth cutoff function \( \chi:(-r,r)\to[0,1] \) which is identically one in a neighborhood of \(0\) and vanishes near \(\{-r,r\}\).  For \(a\in\overline{B}_\rho^k\), define a vector field on \(K_r\) by \( X_a := \chi(\tau)\,(u_a\circ\pi)\,Z, \) and extend it by zero outside \(K_r\). By $Z$ tangent to $\partial N$, hence \(X_a\) is tangent to \(\partial N\). Let \(\Psi_a\) be the time-one flow of \(X_a\), it also preserves $\partial N$. Moreover, \( \frac{\partial X_{a}}{\partial a_i} |_{a=0}= (\zeta^{-1}\varphi_i)Z = \varphi_i\nu_\Sigma \).

After decreasing \(\rho\), all trajectories starting from \(\Sigma\) remain in the region where \(\chi\equiv1\).  Since \(Z(\tau)=1\) and \(Z(\pi)=0,\) we obtain \( \Psi_a(F(x,0)) = F(x,u_a(x)). \) Consequently,
\[
 \gamma_a
 =
 (\Psi_a)_\#\llbracket\Sigma\rrbracket
 =
 \llbracket
 \{F(x,u_a(x)):x\in\Sigma\}
 \rrbracket
\]
agrees with the standard $k$-disk constructed in Section \ref{sec: standart kfamily}, and in particular \(s(\gamma_a)=\bar s(\gamma_a)=a. \)

Because \(X_{-a}=-X_a\), uniqueness of flows gives \( \Psi_{-a}=\Psi_a^{-1}. \) Moreover, after decreasing \(\rho\) once more, \( \|\Psi_a-\operatorname{Id}\|_{C^1}<\beta \) uniformly for \(a\in\overline{B}_\rho^k\).

For the area function \(A^\Sigma(a)=\mathbf M(\gamma_a), \) the second variation formula gives
\[
 \frac{\partial^2 A^\Sigma}{\partial a_i\partial a_j}(0)
 =
 Q(\varphi_i,\varphi_j)
 =
 \lambda_i\delta_{ij},\qquad\text{where }\lambda_i<0,\quad i=1,\ldots,k.
\]
Hence \(D^2A^\Sigma(0)\) is negative definite.  By continuity of the Hessian, after decreasing \(\rho\), there exists \(c_0>0\) such that
\[
        D^2A^\Sigma(a)(u,u)
        \leq-c_0|u|^2
\]
for all \(a\in\overline{B}_\rho^k\) and \(u\in\mathbb R^k\). Therefore \(A^\Sigma\) is strictly concave.

It remains to prove the local min--max inequality.  For \(S\in\mathcal Z_n(N,\partial N;\mathbb Z)\), define
\[
 G_S:\overline{B}_\rho^k\longrightarrow\mathbb R^k,
 \qquad
 G_S(a):=\bar s\bigl((\Psi_a)_\#S\bigr).
\]
When \(S=\llbracket\Sigma\rrbracket\), we have \( G_{\llbracket\Sigma\rrbracket}(a) = s(\gamma_a) = a. \)

For general \(S\), we compute
\[
\begin{aligned}
 G_S(a)-a
 &=
 \bar s\bigl((\Psi_a)_\#S\bigr)
 -
 \bar s\bigl((\Psi_a)_\#
 \llbracket\Sigma\rrbracket\bigr).
\end{aligned}
\]
For each component, this gives \( \bigl(G_S(a)-a\bigr)_i = (S-\llbracket\Sigma\rrbracket)(\Psi_a^*\bar \eta_i). \) Since every \(\Psi_a\) preserves \(A_r\), the form \(\Psi_a^*\bar \eta_i\) vanishes on \(T A_r\).  Moreover,
\[
 \sup_{a\in\overline{B}_\rho^k}
 \left(
 \|\Psi_a^*\bar\eta_i\|_{C^0}
 +
 \|d(\Psi_a^*\bar\eta_i)\|_{C^0}
 \right)
 <\infty.
\]
It follows from the definition of the $\mathcal F_{rel}$-norm and \eqref{eq: LIPcontinuity} that
\[
 \sup_{a\in\overline{B}_\rho^k}
 |G_S(a)-a|
 \leq
 C\mathcal F_{\mathrm{rel}}
 (S-\llbracket\Sigma\rrbracket)
\]
for a constant \(C\) independent of \(S\).

By \( (\Psi_a)_\#S-\llbracket\Sigma\rrbracket = (\Psi_a)_\# (S-\llbracket\Sigma\rrbracket)+ \bigl( (\Psi_a)_\#\llbracket\Sigma\rrbracket -\llbracket\Sigma\rrbracket \bigr)= (\Psi_a)_\# (S-\llbracket\Sigma\rrbracket) + \gamma_a-\llbracket\Sigma\rrbracket \), we have \[\mathcal F_{\mathrm{rel}} \bigl((\Psi_a)_\#S-\llbracket\Sigma\rrbracket\bigr)\le \mathcal F_{\mathrm{rel}} \bigl((\Psi_a)_\# (S-\llbracket\Sigma\rrbracket)\bigr) + \mathcal F_{\mathrm{rel}} (\gamma_a-\llbracket\Sigma\rrbracket)\le C_\Psi \mathcal F_{\mathrm{rel}} (S-\llbracket\Sigma\rrbracket) + C_\gamma|a| \]and thus uniformly,
\[
\sup_{|a|\le\rho}
\mathcal F_{\mathrm{rel}}
\bigl((\Psi_a)_\#S-\llbracket\Sigma\rrbracket\bigr)
\le
C_\Psi
\mathcal F_{\mathrm{rel}}
(S-\llbracket\Sigma\rrbracket)
+
C_\gamma\rho
\] where $C_\Psi=\sup_a\max\{(\operatorname{Lip}\Psi_a)^n,(\operatorname{Lip}\Psi_a)^{n+1}\}<\infty$ and $C_\gamma$ guaranteed by \eqref{eq: Jn+1H}. Thus, after first shrinking $\rho$ such that $\rho<\frac{\varepsilon}{2C_\gamma}$ and then choose \(\varepsilon_0>0\) sufficiently small that 
\(
        \varepsilon_0<\operatorname{min}\{\frac{\rho}{2C},\frac{\varepsilon}{2C_\Psi}\},
\) where $\varepsilon$ is the constant in Corollary \ref{cor: two enhance}. 
Then, whenever 
\(
        \mathcal F_{\mathrm{rel}}
        (S-\llbracket\Sigma\rrbracket)<\varepsilon_0,
\) 
we have
\[
        |G_S(a)-a|<\frac{\rho}{2}
        \qquad
        \text{for every }a\in\partial B_\rho^k.
\]
In particular, the homotopy \( H(t,a):=(1-t)a+tG_S(a) \) does not vanish on \(\partial B_\rho^k\). Therefore
\[
 \deg(G_S,B_\rho^k,0)
 =
 \deg(\operatorname{Id},B_\rho^k,0)
 =
 1.
\]
Consequently, there exists \(a_S\in B_\rho^k\) such that \( G_S(a_S)=0, \) or equivalently \( \bar s\bigl((\Psi_{a_S})_\#S\bigr)=0. \)

Since \(\Psi_{a_S}\) is isotopic to the identity through boundary-preserving diffeomorphisms,
\[
 [(\Psi_{a_S})_\#S]
 =
 [S]
 =
 [\llbracket\Sigma\rrbracket]
 \quad\text{in }H_n(N,\partial N;\mathbb Z).
\]
The choice of $\varepsilon_0$ and Corollary \ref{cor: two enhance} therefore give
\[
\begin{aligned}
 \mathbf M\bigl((\Psi_{a_S})_\#S\bigr)
 &=
 \mathbf M^*\bigl((\Psi_{a_S})_\#S\bigr)\geq
 \mathbf M^*(\llbracket\Sigma\rrbracket)=
 \mathbf M(\Sigma).
\end{aligned}
\]

Finally, suppose that equality holds. Then \(s\bigl((\Psi_{a_S})_\#S\bigr)=0, \) The strict equality case in Corollary \ref{cor: two enhance} implies \( (\Psi_{a_S})_\#S = \llbracket\Sigma\rrbracket. \) This proves the theorem.
\end{proof}

\begin{corollary}
\label{cor:robust-concavity}
Define \(A^V(a):=\|(\Psi_a)_\#V\|(N).\) After decreasing \(\rho\), there exist \(s>0\) and constants \(0<c_0<c_1\) such that, whenever \( \mathbf F_V(V,|\Sigma|)<s, \) one has
\begin{equation}
 -c_1|u|^2
 \le D^2A^V(a)(u,u)
 \le-c_0|u|^2
\label{eq:uniform-concavity}
\end{equation}
for all \(a\in\overline B_\rho^k\) and \(u\in\mathbb R^k\). Moreover, \(A^V\) has a unique maximizer \( m(V)\in B_\rho^k, \) and \(V\mapsto m(V)\) is continuous in the varifold topology.
\end{corollary}
\begin{proof}
If \(J\Psi_a(x,P) \) denotes the \(n\)-Jacobian of \(D\Psi_a(x)\) on the plane \(P\), then
\[
 A^V(a)=\int_{G_n(N)}J\Psi_a(x,P)\,dV(x,P).
\]
The family \(\Psi_a\) is smooth in \(a\), and its support is contained in a fixed compact subset.  Hence the functions \(\partial_{a_i}\partial_{a_j}J\Psi_a(x,P)\) are continuous and uniformly bounded on \(\overline B_\rho^k\times G_n(N)\). It follows that \( (V,a)\mapsto D^2A^V(a) \) is continuous, uniformly in \(a\), when \(V\) is equipped with the varifold topology and has uniformly bounded mass.

For \(V=|\Sigma|\), Theorem \ref{thm:free-boundary-local-minmax} gives a uniformly negative Hessian. Compactness of \(\overline B_\rho^k\) therefore gives \eqref{eq:uniform-concavity} after decreasing \(s\). By decreasing \(\rho\), we may also arrange that
\[
 \nabla A^{|\Sigma|}(0)=0
 \quad\text{and}\quad
 \nabla A^{V}(a)\cdot a<0\qquad
|a|=\rho,\,\mathbf F_V(V,|\Sigma|)<s.
\]
Thus \(A^V\) attains its maximum in the interior. Strict concavity makes the maximizer unique.

If \(V_i\to V\), compactness gives a subsequence \(m(V_i)\to b\in\overline B_\rho^k\). Passing to the limit in \(\nabla A^{V_i}(m(V_i))=0\) yields \(\nabla A^V(b)=0\). Uniqueness gives \(b=m(V)\), hence the whole sequence converges.
\end{proof}

\subsection{Application to the First Free-boundary Width}

We briefly recall the min--max notation. We do not need the full construction of the free-boundary min-max theory, and only use its standard consequences. For full construction and related results, see \cite{FBMHI,FBMHII,FBMH1}.

Let \( \mathcal Z_n^0(N,\partial N;\,\cdot\,;\mathbb Z_2) \) denote the space of relative \(n\)-boundaries equipped with $\mathbf F$ or $\mathcal F_{rel}$--topology. Here $\mathbf F(T_1,T_2)=\mathcal F_{rel}(T_1,T_2)+\mathbf{F}_V(|T_1|,|T_2|)$. A $p$-sweepout is a continuous map \( \Phi:X\to \mathcal Z_n^0(N,\partial N;\mathcal F_{rel};\mathbb Z_2) \) from a finite cubical complex \(X\) such that: \begin{enumerate}
\item $\Phi$ has no concentration of mass.
\item $\Phi^*(\lambda^p)\neq 0$ in $H^p(X;\mathbb Z_2)$.
\end{enumerate}

Here $\lambda\in H^1(\mathcal Z^0_n(N,\partial N;\mathbb Z_2);\mathbb Z_2)$ denotes the fundamental cohomology class given by the Almgren isomorphism \cite{ALM} (See also \cite[Section 2.5]{WEYLLAW}). In some works, the requirement that \(\Phi\) be \(\mathcal F_{{rel}}\)-continuous and have no concentration of mass is replaced by the stronger requirement that \(\Phi\) be \(\mathbf F\)-continuous.

The first free-boundary width is defined by \( \omega_1(N,g) := \inf_{\Phi} \sup_{x}\mathbf M(\Phi(x)), \) where the infimum is taken over all one-sweepouts, namely the class $\mathcal P_1$. A minimizing sequence for \(\omega_1\) is a sequence of one-sweepouts \( \{\Phi_j:X\to\mathcal Z^0_n(N,\partial N;\mathbb Z_2)\}_{j=1} \) such that \( \lim_{j\to\infty}\sup_{x}\mathbf M(\Phi_j(x)) = \omega_1(N,g). \) The associated critical set consists of stationary integral varifolds obtained as varifold limits of slices \( |\Phi_j(x_j)| \) with \( \mathbf M(\Phi_j(x_j))\to \omega_1(N,g). \)

We shall use the following standard consequences of the free-boundary min--max theory. By the strong bumpy metric theorem \cite[Theorem 1.3]{FBMH1}, for a generic metric, every embedded free-boundary minimal hypersurface is proper and non-degenerate. Thus under this assumption, embedded free-boundary minimal hypersurfaces produced by the min--max theory are automatically properly embedded. Moreover, by the multiplicity-one theorem in same article \cite[Theorem 1.1]{FBMH1}, for a generic metric in dimensions \(3\le n+1\le 7\), there exists a two-sided, properly embedded, multiplicity-one free-boundary minimal hypersurface \( \Gamma \) realizing \(\omega_1(N^{n+1},g)\). By the Morse index upper bound of \cite{FBMHII}, such a multiplicity-one critical hypersurface satisfies \( \operatorname{index}(\Gamma)\le 1. \) Note that we use the convention
\[
\operatorname{index}(\Gamma)
:=
\sum_{\ell=1}^q \operatorname{index}(\Gamma_\ell),\qquad \Gamma=\Gamma_1+\cdots+\Gamma_q
\]
if \(\Gamma\) is disconnected.

For later use, tracing \cite{FBMH1}, we could first establish the following lemma.
\begin{lemma}\label{lemmaSWZ}
Let \(3\le n+1\le 7\), and let \((N^{n+1},g)\) be a compact Riemannian manifold with boundary. Assume that the bumpy metric \(g\) is chosen generically so that the multiplicity-one theorem and strong bumpy metric theorem \cite[Theorem 1.1, 1.3]{FBMH1} holds.

Denote \(L_k:=\omega_k(N,g).\) Then there exists a connected finite cubical complex \(X\) and a minimizing sequence of $k$-sweepouts $\{\Phi_j:X\to\mathcal Z^0_n(N,\partial N;\mathcal F_{rel};\mathbb Z_2)\}_j$, points $x_j\in X$, such that \begin{enumerate}
\item $\mathbf M(\Phi_{j}(x_j))\to L_k.$
\item $|\Phi_{j}(x_j)|\to |\Gamma|$ as multiplicity one varifold convergence.
\item $\mathcal F_{rel}(\Phi_j(x_j)-\llbracket\Gamma \rrbracket)\to 0.$
\end{enumerate} 
where $\Gamma$ is a smooth, properly embedded, two-sided, multiplicity-one free-boundary minimal hypersurface with $\mathbf M(\Gamma )=L_k$, $\operatorname{Index}({\Gamma})\le k.$ Note that 2 and 3 together gives $\mathbf F_{V}(|\Phi_j(x_j)|,|\Gamma |)\to 0.$
\end{lemma}
\begin{proof}1 and 2 are standard for minimizing sequences under min--max setting, we only need to show 3. We shall heritage all notations in \cite{FBMH1}.

\medskip
\noindent
\textbf{Step 1.} As in the proof of \cite[Theorem 4.7]{FBMH1}, fix $k$ and ignore subscript $k$ below, we could pick a continuous $k$-sweepout $\Phi_0:X\to \mathcal Z^0_n(N,\partial N;\mathbf F;\mathbb Z_2)$ and denote its $\mathcal F_{rel}$--homotopy class by $[\Phi_0]=\Pi$, with $L(\Pi):=\inf_{\Phi\in \Pi}\sup_{x}\mathbf M(\Phi(x))=\omega_k(N,g)=L$. Here we may assume that $X$ is connected. We could also get a minimizing sequence $\{\Phi_i:X\to\mathcal Z^0_n(N,\partial N;\mathbb Z_2)\}_i$ in the fixed class $\Pi$, deformed as in \cite[Proof of Theorem 4.7, Step I, II]{FBMH1}.

Repeat the argument in \cite[Proof of Theorem 4.7, Step I, II]{FBMH1}, for all sufficient large $i$, would give:\begin{itemize}
\item Subcomplex $\widetilde Y_i\subset X$ and boundary thickening subcomplex $\widetilde{\mathbf B}_i\subset \widetilde Y_i$
\item A lifting to Caccioppoli sets $\widetilde \Phi_i:\widetilde Y_i\to \mathcal C(N)$ such that $\partial \widetilde \Phi_i=\Phi_i|_{\widetilde Y_i} $
\item The associated $(\widetilde Y_i,\widetilde{\mathbf B}_i)$-homotopy class $[\widetilde \Phi_i]=:\widetilde \Pi_i$
\end{itemize}
Furthermore, from \cite[Lemma 4.8]{FBMH1} we also have $\operatorname{max}_{\widetilde{\mathbf B}_i}\mathbf M(\partial\widetilde\Phi_i(x))<L\le L(\widetilde\Pi_i)\to L$. In particular, for $i$ sufficiently large, have $L= L(\widetilde\Pi_i)$ \cite[Proof of Theorem 4.7, Step III]{FBMH1}. We fix a such $i=i_0$ and ignore the subscript $i$, by \cite[Proof of Theorem 4.7, Claim 5]{FBMH1}, $\exists \delta_0>0$, such that $\operatorname{max}_{\widetilde{\mathbf B}}\mathbf M(\partial\widetilde\Phi(x))\le L-4 \delta_0 $ and $\operatorname{max}_{X\setminus\widetilde{Y}}\mathbf M(\Phi(x))\le L-4 \delta_0 $.

\medskip
\noindent
\textbf{Step 2.} Let $\varepsilon_j\to 0$ be a sequence of positive numbers, and choose $h$ as in \cite[Proof of Theorem 4.1]{FBMH1}, such that $\varepsilon_j h\in\mathcal S(g)$ for each $j.$ Here $\mathcal S(g)$ denotes the class of Morse functions satisfying the genericity assumptions with respect to metric $g$ (For precise definition, see \cite[Section 2.2]{FBMH1}). For a Caccioppoli set $E\in \mathcal C(N),$ consider the functional \(\mathcal A^{\varepsilon_j h}(E):=\mathbf M(\partial E)-\varepsilon_j\int_E h\) and denote by $L_j:=L^{\varepsilon_j h}(\widetilde\Pi)$ the corresponding width. Consequently \[|\mathcal A^{\varepsilon_j h}(E)-\mathbf M(\partial E)|\le \varepsilon_j\|h\|_{C^0(N)}\operatorname{Vol}(N), \qquad |L_j-L(\widetilde{\Pi})|\le \varepsilon_j\|h\|_{C^0(N)}\operatorname{Vol}(N).\] Hence $L_j\to L $ and $L_j>\operatorname{max}_{\widetilde{\mathbf B}}\mathcal A^{\varepsilon_j h}(\widetilde\Phi(x))$ for $j$ sufficient large. Thus, \cite[Theorem 3.11]{FBMH1} applies to the class $\widetilde{\Pi}, $ and hence for each fixed $j$ the theorem produces an almost embedded free-boundary $\varepsilon_j h$--hypersurface $\Sigma_j=\partial \Omega_j$, where $\Omega_j\in\mathcal C(N).$ Moreover we also have $\mathcal A^{\varepsilon_j h}(\Omega_j)=L_j$ and weak Morse index $\operatorname{Index}_\omega{\Sigma_j}\le k$ (For precise definition of weak Morse index, see \cite[Definition 2.6]{FBMH1}).

Now we trace \cite[Proof of Theorem 3.11, Step B, C]{FBMH1} for further properties. For each fixed $j$, after taking the Almgren extension of discrete maps appearing in Step B, there exists a minimizing sequence $\widetilde{\Psi}_{j,r}:\widetilde{Y}\to \mathcal C(N)$, points $x_{j,r}\in \widetilde{Y}$ and associated Caccioppoli sets $E_{j,r}:= \widetilde{\Psi}_{j,r}(x_{j,r})$ such that as $r\to\infty$:
\begin{itemize}
\item $\operatorname{sup}_{\widetilde{Y}} \mathcal A^{\varepsilon_j h}(\widetilde{\Psi}_{j,r}(x))\to L_j$
\item $\mathcal A^{\varepsilon_j h}(E_{j,r})\to L_j$
\item $\operatorname{Vol}(E_{j,r}\triangle \Omega_j)\to 0$
\item $\mathbf F_V(|\partial E_{j,r}|,|\Sigma_j|)\to 0$
\end{itemize}
Note that last two convergence come from Step C. In particular, the argument there rules out cancellation and proves that the varifold obtained in Step B is exactly $|\partial \Omega_j|$.

Since $\{\widetilde{\Psi}_{j,r}\}_r$ belong to the class $\widetilde{\Pi}$, there exist homotopies
\[
\begin{array}{c@{\qquad}c}
\vcenter{\hbox{$\widetilde{H}_{j,r}:[0,1]\times \widetilde{Y}\to\mathcal C(N)$,\qquad with }}
&
\begin{cases}
\widetilde{H}_{j,r}(0,\,\cdot\,)=\widetilde{\Phi}\\
\widetilde{H}_{j,r}(1,\,\cdot\,)=\widetilde{\Psi}_{j,r}\\
\sup\limits_{[0,1]\times \tilde{\mathbf B}}\mathbf F_{\mathcal C} (\widetilde{H}_{j,r}(t,x),\widetilde{\Phi}(x))\xrightarrow{r\to \infty} 0
\end{cases}
\end{array}
\]
where $\mathbf F_{\mathcal C}(E,F)=\mathcal F_{rel}(E-F)+\mathbf F_V(|\partial E|,|\partial F|)$.

Now we could perform a diagonal selection. For each $j$, choose $r(j)$ sufficiently large such that \begin{itemize}
\item $\operatorname{sup}_{\tilde{Y}}\mathcal A^{\varepsilon_j h}(\widetilde{\Psi}_{j,r(j)})\le L_j+\frac1j$
\item $|\mathcal A^{\varepsilon_j h}(E_{j,r(j)})-L_j|<\frac1j$
\item $\operatorname{Vol}(E_{j,r(j)}\triangle \Omega_j)<\frac1j$
\item $\mathbf F_V(|\partial E_{j,r(j)}|,|\Sigma_j|)<\frac1j$
\item $\operatorname{sup}_{[0,1]\times \tilde{\mathbf B}}\mathbf F_{\mathcal C} (\widetilde{H}_{j,r(j)}(t,x),\widetilde{\Phi}(x))<\frac1j$
\end{itemize}
For simplicity, for the rest of proof we write $\widetilde{\Psi}_{j}:=\widetilde{\Psi}_{j,r(j)},\,\widetilde{H}_{j}:=\widetilde{H}_{j,r(j)},\,E_{j}:=E_{j,r(j)},\,x_j:=x_{j,r(j)}.$ Recall that $L_j\to L$ would give $\operatorname{sup}_{ \widetilde{Y}}\mathbf M(\partial \widetilde{\Psi}_{j}(x))\le L+o(1)$, $\mathbf M(\partial E_j)\to L,$ and for $j $ sufficiently large, $\operatorname{sup}_{[0,1]\times \tilde{\mathbf B}}\mathbf M(\partial\widetilde{H}_{j}(t,x))<L-3\delta_0$. This implies $x_j\notin \widetilde{\mathbf B}$ for $j$ sufficiently large.

\medskip
\noindent
\textbf{Step 3.} By \cite[Theorem 2.9]{FBMH1} and \cite[Proof of Theorem 4.1]{FBMH1}, after passing to a subsequence, $|\Sigma_j|\to|\Gamma|$ as varifolds, where $\Gamma$ is a smooth, properly embedded, two-sided, multiplicity-one free-boundary minimal hypersurface satisfying $\mathbf M(\Gamma )=L,\,\operatorname{Index}({\Gamma})\le k.$ By BV compactness, after passing to a further subsequence, there exists $\Omega\in\mathcal C(N)$ such that $\Omega_j\to \Omega$ in $L^1(N) $ (Here we identify a Caccioppoli set and its characteristic function). The convergence $\Sigma_j\to \Gamma$ is locally smooth and graphical away from a finite set $\mathcal Y\subset \Gamma.$ Since the limiting multiplicity is one, the convergence consists of a single graph over $\Gamma\setminus\mathcal Y.$ It follows that $\partial \Omega =\llbracket\Gamma\rrbracket$ in $\mathcal Z^0_n(N,\partial N;\mathbb Z_2).$ Recall $\operatorname{Vol}(E_{j}\triangle \Omega_j)<\frac1j$ and $\mathbf F_V(|\partial E_{j}|,|\Sigma_j|)<\frac1j$ give that \begin{itemize}
\item $\mathcal F_{rel}(\partial E_j-\llbracket\Gamma\rrbracket)\le \operatorname{Vol}(E_{j}\triangle \Omega_j)+ \operatorname{Vol}(\Omega\triangle \Omega_j)\to 0$
\item $\mathbf F_V(|\partial E_{j}|,|\Gamma|)\to 0$
\end{itemize}
Thus $\partial E_{j}$ simultaneously converge both in $\mathcal F_{rel}$ and varifold topologies.

\medskip
\noindent
\textbf{Step 4.} In this step we use the cutoff gluing argument in \cite[Proof of Lemma 4.8]{FBMH1}. Since $\widetilde{\mathbf B}$ is a boundary thickening, we may choose a continuous function \[\rho:\widetilde{Y}\to [0,1]\qquad\text{such that }\begin{cases} \rho=1 & \text{on } \widetilde{Y}\setminus \widetilde{\mathbf B}\\ \rho=0 & \text{along the interface between } \widetilde{Y} \text{and } X\setminus\widetilde{Y}
\end{cases}
\]
Define \[ H_{j}(t,x)=\begin{cases} \Phi(x) & x\in X\setminus\widetilde{Y}\\ \partial \widetilde{H}_j(t\rho(x),x)&x\in \widetilde{Y}
  \end{cases}
  \]
Along the gluing interface, $\rho=0$, and hence $\partial \widetilde{H}_j(t\rho(x),x)=\partial \widetilde{H}_j(0,x)=\partial \widetilde{\Phi}(x)= \Phi(x)$. Therefore $H_j$ is a $\mathcal F_{rel}$-continuous homotopy. Set $\widehat{\Psi}_j:=H_j(1,\,\cdot\,)$, then $\widehat{\Psi}_j\simeq_{\mathcal F_{rel}}\Phi$, and hence represents the same nontrivial $k$-sweepout class. Recall that $\operatorname{max}_{X\setminus\widetilde{Y}}\mathbf M(\Phi(x))\le L-4 \delta_0 $ and $\operatorname{sup}_{[0,1]\times \tilde{\mathbf B}}\mathbf M(\partial\widetilde{H}_{j}(t,x))<L-3\delta_0$ would have:
\begin{itemize}
\item $\mathbf M(\widehat \Psi_j(x))\le L-4 \delta_0 $ on $X\setminus\widetilde{Y}$
\item $\mathbf M(\widehat \Psi_j(x))\le L-3 \delta_0 $ on $\widetilde{\mathbf B}$
\item $\widehat{\Psi}_j(x)=\partial \widetilde{\Psi}_j(x)$ on $\widetilde{Y}\setminus\widetilde{\mathbf B}$
\end{itemize}
By $\operatorname{sup}_{ \widetilde{Y}}\mathbf M(\partial \widetilde{\Psi}_{j}(x))\le L+o(1)$, consequently $\operatorname{sup}_{ X}\mathbf M( \widehat{\Psi}_{j}(x))\le L+o(1)$. Exactly as in the last paragraph of \cite[Proof of Lemma 4.8]{FBMH1}, the $\mathbf F_{\mathcal C}$-control on $\widetilde{\mathbf B}$ implies that the sequence $\{\widehat{\Psi}_j\}$ has no concentration of mass. Hence each $\widehat{\Psi}_j$ is a $k$-sweepout and $\operatorname{sup}_{ X}\mathbf M( \widehat{\Psi}_{j})\ge L$. Thus $\operatorname{sup}_{ X}\mathbf M( \widehat{\Psi}_{j})\to L$.

By Step 2, $x_j\notin \widetilde{\mathbf B}$, $\rho(x_j)=1$. Therefore the gluing leaves the selected slice unchanged, i.e. $\widehat{\Psi}_j(x_j)=\partial \widetilde{H}_j(1,x_j)=\partial\widetilde{\Psi}_j(x_j)=\partial E_j$. Combining this with convergence result in Step 3, using triangle inequality we have\begin{itemize}
\item $\mathcal F_{rel}(\widehat{\Psi}_j(x_j)-\llbracket\Gamma\rrbracket)\to 0$
\item $\mathbf F_V(|\widehat{\Psi}_j(x_j)|,|\Gamma|)\to 0$
\end{itemize} 
Relabel $\widehat{\Psi}_j$ as $\Phi_j$ would complete the proof.\end{proof}

\begin{theorem}
\label{thm:index-one-omega-one}
Under the hypotheses of Lemma~\ref{lemmaSWZ}, there exists a smooth, properly embedded, two-sided, multiplicity-one free-boundary minimal hypersurface \(\Gamma\) such that
\[
\mathbf M(\Gamma)=\omega_1(N,g)
\qquad\text{and}\qquad
\operatorname{Index}(\Gamma)=1.
\]

\end{theorem}

\begin{proof}
Fix $k=1$ and let \( \Gamma,\, \{\Phi_j\},\,\{x_j\} \) be furnished by Lemma~\ref{lemmaSWZ}. In particular, \(\operatorname{index}(\Gamma)\leq1\). It remains to prove that \( \operatorname{index}(\Gamma)\ne0. \) Suppose by contradiction that \( \operatorname{index}(\Gamma)=0. \) Then every component \(\Gamma_\ell\) is stable. Since \(g\) is bumpy, no component admits a nontrivial free-boundary Jacobi field. Hence each \(\Gamma_\ell\) is strictly stable, and therefore the whole multiplicity-one cycle \( \Gamma \) is strictly stable.

By Corollary~\ref{cor: two enhance}, there exists a $\mathcal F_{rel}$--neighborhood \( \mathcal U \) of \( \llbracket\Gamma\rrbracket \) such that \(\llbracket\Gamma\rrbracket\) is the unique mass minimizer in its relative homology class inside \(\overline{\mathcal U}^{\,\mathcal F_{rel}}\). By compactness and lower semicontinuity, we may choose another small $\mathcal F_{rel}$-neighborhood \( \mathcal U'\Subset\mathcal U\) and \(\delta>0 \) such that
\[
T\in \mathcal U\setminus \mathcal U'
\quad\Longrightarrow\quad
\mathbf M(T)\ge L+2\delta.
\] Shrinking \(\mathcal U\) if necessary, by $\mathcal F_{rel}$-isoperimetric lemma \cite[Lemma 3.15]{FBMHI}, we may assume that every relative cycle in \(\mathcal U\) represents the same relative homology class. 
Since 
\(
\{\Phi_j\}
\) 
is minimizing for \(\omega_1(N,g)\), for all sufficiently large \(j\) we have 
\(
\sup_{x}\mathbf M(\Phi_j(x))<L+\delta.
\) 
We claim that, for all sufficiently large \(j\),
 \(
\Phi_j(X)\cap\mathcal U'=\varnothing.
\) 
\begin{proof}[Proof of Claim]

Suppose not, then there exist \(j\) and \(x_0\in X\) such that 
\(
\Phi_j(x_0)\in\mathcal U'.
\) 
Since \(X\) is connected, either 
\(
\Phi_j(X)\subset\mathcal U
\) 
or else \(\Phi_j(X)\) intersects 
\(
\mathcal U\setminus\mathcal U'.
\) 
In the second case, there exists \(x_1\in X\) such that 
\(
\Phi_j(x_1)\in\mathcal U\setminus\mathcal U'.
\) 
Hence 
\(
\mathbf M(\Phi_j(x_1))\ge L+2\delta,
\) 
contradicting 
\(
\sup_{x}\mathbf M(\Phi_j(x))<L+\delta.
\) 
Thus the only remaining possibility is 
\(
\Phi_j(X)\subset\mathcal U.
\) 

But by our choice of \(\mathcal U\), it is trivial for the first sweepout class, i.e. the Almgren cohomology class detecting one-sweepouts vanishes on \(\mathcal U\) \cite[Proposition 2.12]{WEYLLAW}. Therefore a map whose image is contained in \(\mathcal U\) cannot represent a nontrivial one-sweepout. This contradicts the fact that \(\Phi_j\) belongs to the minimizing sequence for \(\omega_1(N,g)\). The claim follows.\end{proof}

On the other hand, by Lemma \ref{lemmaSWZ}, \(\Phi_{j}(x_j)\in\mathcal U'\) for \(j\) sufficiently large, contradicting the claim. Therefore 
\(
\operatorname{index}(\Gamma)=1.
\)\end{proof}



The preceding application should be regarded as a first use of the $\mathcal F_{rel}$-neighborhood local minimality theorem. We expect that the theorem may be useful in other geometric variational problems.

\paragraph{Conflit of Interest}
On behalf of all authors, the corresponding author states that there is no conflict of interest.

\paragraph{Data Availability}
Data sharing is not applicable to this article, as no datasets were generated or analyzed during the current study.

\paragraph{Funding}
The authors are supported by the National Natural Science Foundation of China (Grant no. 12371055). 

\begin{flushleft}
\medskip\noindent
\begin{tabbing}

				Xiaoxiang Jiao\\
				School of Mathematical Sciences, University of Chinese Academy of Sciences\\
				19A Yuquan Road, Beijing, 100049, China\\
				\texttt{xxjiao@ucas.ac.cn}
				
\end{tabbing}

\begin{tabbing}
Qinhan Zhao\\School of Mathematical Sciences, University of Chinese Academy of Sciences\\
				19A Yuquan Road, Beijing, 100049, China\\
				\texttt{zhaoqinhan22@mails.ucas.ac.cn}
\end{tabbing}

\begin{tabbing}
Hangyue Zhu\\School of Mathematical Sciences, University of Chinese Academy of Sciences\\
				19A Yuquan Road, Beijing, 100049, China\\
				\texttt{zhuhangyue24@mails.ucas.ac.cn}
\end{tabbing}

\end{flushleft}

\bibliographystyle{alpha}
\bibliography{ref}

\end{document}